\documentclass[12pt,a4paper]{article}
\usepackage[english]{babel}
\usepackage{authblk}
\usepackage{amsmath,amssymb,amsfonts,amsthm,bm}
\usepackage{geometry}
\usepackage{booktabs}
\usepackage{float}
\usepackage{graphicx}
\usepackage{subcaption}
\usepackage{placeins}
\usepackage{multirow,multicol}
\usepackage{natbib}
\usepackage[colorlinks, citecolor=blue, linkcolor=black]{hyperref}
\numberwithin{equation}{section}
\allowdisplaybreaks
\newtheorem{theorem}{Theorem}[section]
\newtheorem{lemma}{Lemma}[section]

\newtheorem{assumption}{Assumption}[section]
\newtheorem{remark}{Remark}[section]

\newcommand{\Var}{\operatorname{Var}}

\title{Error analysis and optimal sample allocation for randomized unbiased estimation: the single term estimator}
\author{Jiayu He}
\author{Chao Zheng \thanks{Corresponding author: chao.zheng12@gmail.com.}}
\affil{School of Data Sciences, Zhejiang University of Finance and Economics, Hangzhou, China}
\date{}

\begin{document}
\maketitle
\begin{abstract}
Randomized unbiased estimation can be highly efficient for approximating expectations of functionals of stochastic differential equations. However, in practice, the optimal computational efficiency is rarely attained due to errors arising from parameter estimation and truncation. The effects of these errors on computational efficiency have not been investigated. In this article, we study the error analysis and optimal sample allocation for the single term estimator. Specifically, under mild assumptions on the estimated parameters, we develop a unified error analysis that incorporates both sources of error, and establish the relevant convergence rates for the computational efficiency loss. Based on this analysis, we further investigate the optimal allocation of the computational budget for parameter estimation, and propose several efficient sample allocation strategies. Numerical experiments support the theoretical convergence rates and demonstrate the effectiveness of our proposed methods.

\end{abstract}

\textbf{Keywords:} the single term estimator, randomized unbiased estimators, stochastic differential equations, error analysis, optimal sample allocation

\textbf{AMS subject classifications (2000):} 60H35, 65C30, 90C34, 91G60

	\section{Introduction}
	
	Monte Carlo methods are fundamental tools for estimating expectations of complex random quantities and have found widespread applications in financial engineering, uncertainty quantification, stochastic control, computational physics and mathematical biology. In many of these applications, the quantities of interest are functionals of solutions to stochastic differential equations (SDEs). When solutions of SDEs can be generated exactly, the standard Monte Carlo estimator is unbiased, and achieves the canonical root-mean-square convergence rate \(O(c^{-1/2})\), where \(c\) is the computational budget. However, exact simulation is generally difficult for SDEs, and time-discrete schemes, such as those in \citep{kloeden1992numerical}, may be used, resulting in biased approximations. According to \citet{duffie1995efficient}, even with an optimal choice of the discretization step size and the number of samples, the root-mean-square error converges at the rate \(O(c^{-r/(2r+1)})\), where \(r\) is the weak convergence rate of the underlying time-discrete scheme. Building on the multilevel Monte Carlo method \citep{giles2008multilevel} and the randomized debiasing idea in \citet{mcleish2011general}, \citet{rhee2015unbiased} made a breakthrough by proposing several randomized unbiased estimators, that eliminate the discretization bias while retaining the canonical square-root convergence rate \(O(c^{-1/2})\).

Since the work of \citet{rhee2015unbiased}, randomized unbiased estimation has been developed and applied in a variety of settings. In the context of SDEs, for example, \citet{andersson2017} considered unbiased estimation using parametrix expansions. \citet{kang2019unbiased} developed unbiased sensitivity estimation of one-dimensional diffusion processes. \citet{vihola2018unbiased} proposed a general class of unbiased estimators, that include those of \citet{rhee2015unbiased} as special cases, as well as several low variance estimators. For Markov chain applications, \citet{glynn2014exact} extended the application to unbiased estimation of expectations of Markov chain equilibrium distributions, and this work was further developed by \citet{jacob2020} and \citet{kahale2024unbiased}, who proposed unbiased Markov chain Monte Carlo methods, and unbiased time-average estimation for Markov chains, respectively.  Another line of research concerns the construction of unbiased estimators for functions of expectations of random variables; see, for example,  \citet{blanchet2015unbiased}, \citet{blanchet2015} and \citet{chopin2025towards}. Regarding asymptotic properties, \citet{zheng2017clt} and \citet{zheng2016rates} established central limit theorems for infinitely stratified estimators and for various debiased Multilevel Monte Carlo algorithms, respectively. Recently, double randomization techniques were developed to address settings involving an additional source of bias that also needs to be eliminated. For instance, \citet{jasra2022unbiased} applied this approach to filtering of partially observed diffusions,  \citet{ruzayqat2023unbiased} employed it to achieved unbiased estimation for a class of diffusion processes, and \citet{heng2023unbiased} developed a general framework for unbiased estimation in discretized probability models.

It is well known that the computational efficiency of unbiased estimators may depend heavily on the choice of probability distributions used for randomization. Despite the popularity of unbiased estimation, the problem concerning the optimal distribution remains largely unexplored. In the context of SDEs, the optimal distribution and the associated algorithms have been investigated by  \citet{rhee2015unbiased}, \citet{cui2021optimal} and \citet{zheng2026optimal}. However, the error arising from the calculation of the optimal distribution has received little attention. To the best of our knowledge, only \citet{zheng2026optimal} have a preliminary discussion on this issue for sum estimators.

In this article, we focus on the single term estimator in the context of SDEs, which is widely used due to its simple structure and easy implementation. However, in practice, the optimal computational efficiency is rarely attained, due to errors arising from parameter estimation and truncation when calculating the optimal distribution. Specifically, the parameters entering the optimal distribution may be estimated by preliminary Monte Carlo simulation, leading to a parameter estimation error. Moreover, since infinitely many parameters are involved, only a finite number of them can be estimated directly, while the remaining tail may be approximated using its asymptotic behavior, thereby introducing a truncation error. These two sources of error jointly determine the computational efficiency loss of the single term estimator. This raises two ‌crucial‌ questions: how does the overall computational efficiency loss depend on the parameter estimation accuracy and the truncation level, and under a fixed computational budget for parameter estimation, how should samples be allocated across different levels to minimize this loss? 

To address these questions, we first develop a probabilistic model for the estimated parameters based on the central limit theorem while preserving their positivity. We then establish a unified error analysis for the computational efficiency loss with respect to both the Monte Carlo sample size and the truncation level. Since the underlying problem is infinite-dimensional, we first analyse an \(m\)-truncated problem and then extend the analysis to the infinite-dimensional setting under a geometric tail assumption. In particular, we derive the convergence rate with respective to the sample size and the truncation level, see Theorem \ref{thm4.4}. Based on this result, we concerntrate on the leading-order term of the parameter estimation error, and propose several sample allocation methods. Numerical experiments show that under a fixed computation budget in prior estimation, the proposed non-uniform allocation methods can substantially reduce the efficiency loss compared with uniform allocation.

Three features of our results are worth highlighting. First, the total error is decomposed into two parts: one arising solely from truncation and the other dominated by the parameter estimation error. Second, the decay rate of the parameters alone is insufficient to determine the convergence rate with respect to the truncation level. The convergence rate also depends on the decay of the non-leading-order term, resulting in a higher convergence rate than would be predicted from the parameter decay alone. Third, among the allocation strategies considered in our numerical experiments, the asymptotic-allocation method is one of the most efficient while requiring only a simple allocation rule, making it attractive for practical implementation. 

	The remainder of this article is organized as follows. Section 2 formulates the problem and introduces assumptions for the estimated parameters. Section 3 establishes the error analysis for the \(m\)-truncated problem. Section 4 extends the analysis to the infinite-dimensional setting and establishes the corresponding convergence rates. Section 5 develops non-uniform sample allocation methods. Section 6 presents numerical experiments, and Section 7 concludes the article.
	
	\section{Problem Formulation}
	
Let $(X(t): t\geq 0)$ be the unique solution to the SDE 
\[
dX(t)=\mu(X(t))dt+\sigma(X(t))dB(t),
\]
where $\mu: \mathbb{R}^d\rightarrow \mathbb{R}^d$, $\sigma: \mathbb{R}^d\rightarrow \mathbb{R}^{d\times l}$ and $(B(t): t\geq 0)$ is an $l$-dimensional Brownian motion. In many applications, one needs to calculate the expectation $E[f(X)]$, where $f$ is a functional of $X$. 

Let $(Y_n: n\geq 0)$ be a sequence of approximations of $Y:=f(X)\in L^2$, such that $E[(Y_n-Y)^2]\rightarrow 0$ as $n\rightarrow \infty$, and  \(\Delta_n = Y_n - Y_{n-1}\) is the increment with \(Y_{-1}=0\). \citet{rhee2015unbiased} proposed the single term estimator
\begin{equation*}
	\bar{Z} = \frac{\Delta_{\mathcal{N}}}{p_{\mathcal{N}}},
\end{equation*}
where \({\mathcal{N}}\) is a nonnegative integer-valued random variable independent of $(Y_n: n\geq 0)$, and \(p_n = P(\mathcal{N}=n)\) is a probability mass function. This estimator is unbiased, i.e., \(E(\bar{Z}) = E(Y)\), provided that \(\sum_{n=0}^{\infty} E(\Delta_n^2)/p_n < \infty\). The efficiency of the single term estimator can be measured by the work-variance product, which is expressed as
	\begin{align}
		\min_p g(p)&=\left(\sum_{n=0}^{\infty}\frac{\beta_n}{p_n} - \alpha^2\right)\left(\sum_{n=0}^{\infty} t_n p_n\right)\label{eq:gp} \\
		\text{s.t. }& p_n>0,\ \forall n\ge 0\nonumber\\
		&\sum_{n=0}^{\infty}p_n = 1, \nonumber
	\end{align}
where \(\beta_n = E(\Delta_n^2)\) is the second moment of the increment, \(t_n = O(2^n)\) is the computational cost to sample $\Delta_n$, and \(\alpha = E(Y)\). It has been proved by \citet{rhee2015unbiased} that the optimal probability mass function
	\begin{equation}
		p_n^* = \frac{\sqrt{\beta_n}}{\sqrt{\alpha^2 + c^* t_n}},\label{eq:pn*}
	\end{equation}
	where \(c^*\) is the unique solution to
	\begin{equation}
	\sum_{n=0}^{\infty}\sqrt{\frac{\beta_n}{\alpha^2 + c t_n}} = 1. \label{eq:c*}
	\end{equation}
Thus, $c^*$ is an implicit function of $(\beta_n,n\geq0)$ and $\alpha$. However, these parameters are generally unknown. 

   In practice, we may estimate a finite number of parameters using Monte Carlo simulation and approximate the remaining tail parameters through their asymptotic behavior. Specifically, we may estimate $\beta_0,\beta_1,...,\beta_m$ via Monte Carlo simulation for a sufficiently large $m$, and then approximate the tail $\beta_{m+1},\beta_{m+2},...$ using a recursion formula $\beta_{n+1}=\beta_n 2^{-2p}$, $\forall n\geq m$, where $p>1/2$. Let
	\[
	\theta := (\beta_0, \ldots, \beta_m, \alpha)^\top.
	\] 
Based on \(N\) independent samples, we construct estimators 
\[
\hat{\beta}_n =\frac{1}{N}\sum_{i=1}^{N}\Delta_{n,i}^2, \quad \hat{\alpha} = \frac{1}{N}\sum_{i=1}^{N}Y_i. 
\]
The estimated parameter vector 
\[
\hat{\theta} := (\hat{\beta}_0, \ldots, \hat{\beta}_m, \hat{\alpha})^\top
\]
is then substituted into equation \eqref{eq:pn*}, together with a truncated version of equation \eqref{eq:c*}, to yield the estimated probability mass function \(\hat{p}_n\). This leads to a computational efficiency loss of the single term estimator, defined as 
\[
g(\hat{p}) - g(p^*)=\left(\sum_{n=0}^{\infty}\frac{\beta_n}{\hat{p}_n} - \alpha^2\right)\left(\sum_{n=0}^{\infty}t_n\hat{p}_n\right)-\left(\sum_{n=0}^{\infty}\frac{\beta_n}{p_n^*} - \alpha^2\right)\left(\sum_{n=0}^{\infty} t_n p_n^*\right).
\]
Since the total error consists of both parameter estimation error and truncation error from the tail approximation, its analysis is challenging.

Before we perform an error analysis, we shall impose an assumption on $\hat{\theta}$. By the central limit theorem, as \(N \to \infty\), the distribution of \(\sqrt{N}(\hat{\theta} - \theta)\) converges to a normal distribution \(\mathcal{N}(0, \Sigma)\), where \(\Sigma\) is the covariance matrix. Moreover, $\hat{\theta}$ must remain positive by the definitions of $(\beta_n,n\geq0)$ and $\alpha$. In financial applications, $\alpha$ represents the option price, and is thus positive. Therefore, we introduce the following assumption on $\hat{\theta}$ to ensure that these properties are satisfied.
	
\begin{assumption}
	\label{assu:1.1}
	The parameter estimators satisfy
	\begin{equation*}
		\hat{\beta}_n = \beta_n + \frac{\sigma_n}{\sqrt{N}} \hat{\varepsilon}_n, \quad \hat{\alpha} = \alpha + \frac{\sigma_{\alpha}}{\sqrt{N}} \hat{\varepsilon}_{m+1},
	\end{equation*}
for any $n=0,..,m$, where $\sigma_n^2 = \operatorname{Var}(\Delta_n^2)$, and $\sigma_{\alpha}^2 = \operatorname{Var}(Y)$. Here, $\beta_n>0$ for all $n$, and $\alpha>0$. The truncated normal random variables $\hat{\varepsilon}_n$ are defined as 
\[
\hat{\varepsilon}_n = 
\begin{cases}
\varepsilon_n, & \text{if } -B_N < \varepsilon_n < B_N, \\
B_N, & \text{if } \varepsilon_n \geq B_N, \\
-B_N, & \text{if } \varepsilon_n \leq -B_N.
\end{cases}
\]
where $\varepsilon_n$, $n=0,..,m+1$ are standard normal random variables, and $B_N$ is a deterministic function of $N$ satisfying $B_N\asymp N^q$, $0<q<1/2$, i.e.,
\[
\lim_{N\rightarrow +\infty}\frac{B_N}{N^q}\in (0,+\infty).
\]
\end{assumption}

\begin{remark}	
We note three points from Assumption \ref{assu:1.1}. First, the variables $\varepsilon_n$ are not necessarily independent, which implies that $\hat{\beta}_n$, $n=0,1,...,m$ and $\hat{\alpha}$ can be correlated. Second, $\hat{\theta}$ is always positive for a sufficiently large $N$, and \(\sqrt{N}(\hat{\theta} - \theta)\) converges in distribution to \(\mathcal{N}(0, \Sigma)\). Third, $\hat{\theta}$ is an unbiased estimator of $\theta$, i.e., $E(\hat{\theta})=\theta$. 
\end{remark}

For notational convenience, let $\hat{\varepsilon}= (\hat{\varepsilon}_0, \dots, \hat{\varepsilon}_{m}, \hat{\varepsilon}_{m+1})\top$, and \(D = \text{diag}(\sigma_0, \ldots, \sigma_m, \sigma_\alpha)\). We can write
\begin{equation}
	\hat{\theta} = \theta + \frac{1}{\sqrt{N}} D\hat{\varepsilon}. \label{eq:deltatheta}
\end{equation}
Moreover, $\|A\|$ denotes the Frobenius norm of the matrix $A=(a_{ij})_{d_1\times d_2}$, such that 
\[
\|A\|=\sqrt{\sum_{i=1}^{d_1}\sum_{j=1}^{d_2}|a_{i,j}|^2}. 
\]
Throughout the article, $\theta$ is fixed, and we write $p_n^*=p_n(\theta)$, $\hat{p}_n=p_n(\hat{\theta})$. Let $\Delta p_n=\hat{p}_n-p_n^*$. Let $\nabla p_n(\hat{\theta})=(\frac{\partial p_n}{\partial \hat{\theta}_0},\frac{\partial p_n}{\partial \hat{\theta}_1},...,\frac{\partial p_n}{\partial \hat{\theta}_{m+1}})$ be the gradient of $p_n$ with respect to $\hat{\theta}$, and $\nabla^2 g_m(p)=(\frac{\partial^2 g_m}{\partial p_i \partial p_j})_{(m+1)\times (m+1)}$ be the Hessian matrix of $g_m(p)$ with respect to $p$. Here, $g_m(p)$ is a truncated version of $g(p)$. Finally, We denote by $C$ a generic constant independent of $m$, whose value may change from line to line, and by $C_m$ a constant that may depend on $m$.

	\section{Finite-Dimensional Error Analysis}
	
	In this section, we analyze the impact of parameter estimation error on the optimal distribution and the computational efficiency loss for the $m$-truncated optimization problem as a starting point. 
	
	The $m$-truncated optimization problem is
	\begin{align}
		\min_p g_m(p)&=\left(\sum_{n=0}^{m}\frac{\beta_n}{p_n} - \alpha^2\right)\left(\sum_{n=0}^{m} t_n p_n\right)\label{eq:gm} \\
		\text{s.t. }& p_n>0,\ \forall n\ge 0\nonumber\\
		&\sum_{n=0}^{m}p_n = 1. \nonumber
	\end{align}
	Following the argument in the proof of Theorem 4 from \citet{rhee2015unbiased}, we obtain the optimal distribution for this truncated problem as
	\begin{equation*}
		p_n^* = \frac{\sqrt{\beta_n}}{\sqrt{\alpha^2 + c^*(m) t_n}},\quad n = 0,1,\dots,m,
	\end{equation*}
	where \(c^*(m)\) is the unique solution of the equation
	\begin{equation}
		\sum_{n=0}^{m}\sqrt{\frac{\beta_n}{\alpha^2 + c t_n}} = 1. \label{eq:hatc}
	\end{equation}

	Our goal is to analyze $E(g_m(\hat{p})-g_m(p^*))$ based on Assumption \ref{assu:1.1}. Since $\hat{\theta}$ converges in probability to $\theta$, a natural approach is to apply a Taylor expansion of $g_m(p(\hat{\theta}))$ around $\theta$. This requires sufficient smoothness and boundedness of the implicit function $c(\hat{\theta})$ and the optimal distribution function $p_n(\hat{\theta})$ in a neighborhood of $\theta$. Therefore, we first establish these regularity properties and derive Lemmas \ref{lem:yijie} and \ref{lem:erjie}.

\begin{lemma}
		\label{lem:yijie}
		Let 
		\[ U_i=[\theta_i-\frac{\sigma_i}{\sqrt{N}}B_N,\theta_i+\frac{\sigma_i}{\sqrt{N}}B_N] \]
		be a bounded closed interval for any $i=0,..,m+1$, where $\sigma_{m+1}:=\sigma_{\alpha}$. Let $U=(U_0,...,U_{m+1})^\top$. Then, for a sufficiently large $N$, it holds that

			\[
			\sup_{\hat{\theta}\in U}\|\nabla c(\hat{\theta})\|<\infty, \quad \sup_{\hat{\theta}\in U}\|\nabla p_n(\hat{\theta})\|<\infty,\qquad \forall n=0,\dots,m,
			\]
			and 
			\[
			\sup_{\hat{\theta}\in U}\|\nabla^2 g_m(p)\|<\infty.
			\]
\end{lemma}
\begin{proof}
	Since $B_N\asymp N^q$, $0<q<1/2$, we have $\hat{\theta}>0$ for all $\hat{\theta}\in U$, when $N$ is sufficiently large. Let 
	\begin{equation}
	 F(\hat{\theta},c)=\sum_{n=0}^{m}\sqrt{\frac{\hat{\beta}_n}{\hat{\alpha}^2 + c t_n}} - 1.\label{F}
   \end{equation}
   Since the left hand side of (\ref{F}) is strictly monotone in $c$, for any $\hat{\theta}\in U$, there exists a unique solution $c(\hat{\theta})$ satisfying $F(\hat{\theta},c)=0$. By the implicit function theorem, we have
\begin{equation}
	\frac{\partial c}{\partial \hat{\beta}_i} = -\frac{1}{2J}\frac{1}{\sqrt{\hat{\beta}_i(\hat{\alpha}^2 + c t_i)}}, \quad \frac{\partial c}{\partial \hat{\alpha}} = \frac{1}{J}\sum_{n=0}^{m}\frac{\hat{\alpha}\sqrt{\hat{\beta}_n}}{(\hat{\alpha}^2 + c t_n)^{3/2}},
			\label{eq:pctheta}
\end{equation}
where 
		\[
		J:=J(\hat{\theta},c(\hat{\theta})) = \frac{\partial F}{\partial c}=-\sum_{n=0}^{m}\frac{t_n\sqrt{\hat{\beta}_n}}{2(\hat{\alpha}^2 + c t_n)^{3/2}}\neq 0.
		\]
These derivatives are continuous functions of $\hat{\theta}$ and $c(\hat\theta)$, and thus are uniformly bounded on the closed interval $U$. Therefore,
\[
\sup_{\hat{\theta}\in U}\|\nabla c(\hat{\theta})\|<\infty.
\]
For the probability mass function $p_n=\sqrt{\hat{\beta}_n}/\sqrt{\hat{\alpha}^2 + c t_n}$, the chain rule gives
\[
\frac{\partial p_n}{\partial \hat{\beta}_i}=\frac{\delta_{n,i}}{2\sqrt{\hat{\beta}_n(\hat{\alpha}^2+ct_n)}}+\frac{1}{4J}\frac{t_n\sqrt{\hat{\beta}_n}}{(\hat{\alpha}^2+ct_n)^{3/2}}\frac{1}{\sqrt{\hat{\beta}_i(\hat{\alpha}^2 + c t_i)}},
\]
where $\delta_{n,i}=0$ if $i\neq n$, and $\delta_{n,i}=1$ otherwise, and 
\[
\frac{\partial p_n}{\partial \hat{\alpha}}=-\frac{\hat{\alpha}\sqrt{\hat{\beta}_n}}{(\hat{\alpha}^2+ct_n)^{3/2}}-\frac{1}{2J}\frac{t_n\sqrt{\hat{\beta}_n}}{(\hat{\alpha}^2+ct_n)^{3/2}}\sum_{n=0}^{m}\frac{\hat{\alpha}\sqrt{\hat{\beta}_n}}{(\hat{\alpha}^2 + c t_n)^{3/2}}.
\]
Analogously, these derivatives are continuous and uniformly bounded on $U$, which implies
		\begin{equation}
			\sup_{\hat{\theta}\in U}\|\nabla p_n(\hat{\theta})\| < \infty.
			\label{eq:Cp}
		\end{equation}
Finally, we consider the objective function $g_m(p)$. A direct differentiation gives
	\begin{equation}
			\frac{\partial^2 g_m}{\partial p_n^2} = \frac{2\beta_n \sum_{i=0}^{m}t_i p_i}{p_n^3} - \frac{2\beta_n t_n}{p_n^2},\qquad
			\frac{\partial^2 g_m}{\partial p_n\partial p_l} = -\frac{t_l\beta_n}{p_n^2} - \frac{t_n\beta_l}{p_l^2}\quad (n\neq l),
			\label{eq:pgpp}
	\end{equation}
which are also uniformly bounded on $U$. Therefore,
		\[
			\sup_{\hat{\theta}\in U}\|\nabla^2 g_m(p)\|<\infty.
		\]
The proof is complete.
\end{proof}
	
	\begin{lemma}
		\label{lem:erjie}
		Under the same setting as in Lemma \ref{lem:yijie}, we have
			\[
			\sup_{\hat{\theta} \in U}\left|\frac{\partial^2c}{\partial\hat{\theta}_i\partial\hat{\theta}_j}\right|<\infty, \quad \sup_{\hat{\theta} \in U}\left|\frac{\partial^2 p_n}{\partial\hat{\theta}_i\partial\hat{\theta}_j}\right|<\infty,\qquad \forall n=0,\dots,m,
			\]
			and
			\[
			\sup_{\hat{\theta} \in U}\left|\frac{\partial^3 g_m}{\partial p_i\partial p_j\partial p_k}\right|<\infty.
			\]
	\end{lemma}
	\begin{proof}
	The proof is analogous to that of Lemma \ref{lem:yijie}, since $J\neq 0$. 
	\end{proof}	
	
	\begin{lemma} \label{lem:optimal}
	It holds that 
	\[
	\nabla g_m(p^*)\Delta p = 0,
	\]
	where $p^*$ denotes the optimal probability mass function for the $m$-truncated optimization problem \eqref{eq:gm}, and $\Delta p:=(p_0-p_0^*,...,p_m-p_m^*)^\top$.
	\end{lemma}
	\begin{proof}
	Let $p_0=1-\sum_{n=1}^{m}p_n$, and define
	\[
	f(p_1,p_2,...,p_m)=g_m(p).
	\]
	Since $p^*$ is optimal, the first-order optimality condition gives
	\[
	\frac{\partial f(p^*)}{\partial p_n}=\frac{\partial g_m(p^*)}{\partial p_n}-\frac{\partial g_m(p^*)}{\partial p_0}=0
	\]
	for any $n=1,..,m$. It follows that 
	\[
	\nabla g_m(p^*)=\frac{\partial g_m(p^*)}{\partial p_0}(1,1,...,1).
	\]
	Moreover, as $p^*$ and $p$ are probability mass functions in the feasible region of optimization problem \eqref{eq:gm}, we have 
	\[
	(1,1,...,1)\Delta p=\sum_{n=0}^{m}(p_n-p_n^*)=0
	\]
	and therefore
	\[
	\nabla g_m(p^*)\Delta p = \frac{\partial g_m(p^*)}{\partial p_0}(1,1,...,1)\Delta p=0.
	\]
	The proof is complete.
	\end{proof}

	\begin{theorem}
		\label{the:yijie}
		Under Assumption \ref{assu:1.1}, the computational efficiency loss $\Delta g = g_m(\hat{p}) - g_m(p^*)$ satisfies
\[
		\displaystyle E(\Delta g)\leq C_m N^{-1}, \quad \displaystyle \Var(\Delta g)\leq C_m N^{-2},
\]
		where $N$ is the sample size.
	\end{theorem}
	\begin{proof}
		Let $\sigma_{\max}=\max(\sigma_0, \ldots, \sigma_m, \sigma_\alpha)\). Equation \eqref{eq:deltatheta} implies that
		\begin{equation}
			\|\Delta\theta\| \le \frac{\sigma_{\max}\|\hat{\varepsilon}\|}{\sqrt{N}}. \label{eq:thetafan}
		\end{equation}
	    By the definition of the truncated normal distribution, we have $E(\hat{\varepsilon}_n^2)<E(\varepsilon_n^2)=1$ for each $n=0,1,...,m+1$, and therefore $E[\|\hat{\varepsilon}\|^2]<C_m$. Hence, we get from \eqref{eq:thetafan} that		
		\[
		E[\|\Delta\theta\|^2] \leq C_m N^{-1},  \quad E[\|\Delta\theta\|^4] \leq  C_m N^{-2}.
		\]
		
		By the mean value theorem, there exists $\xi_n$ between $\theta$ and $\hat{\theta}$, such that
		\[
		\Delta p_n = p_n(\hat{\theta}) - p_n(\theta) = \nabla p_n(\xi_n) \Delta\theta.
		\]
		It follows from the boundedness argument in Lemma \ref{lem:yijie} that 
		\[
			\|\Delta p\| \le C_m\|\Delta\theta\|. 
		\]
		Applying a second-order Taylor expansion of $g_m(p)$ at $p^*$, together with Lemma \ref{lem:optimal}, we obtain
		\[
		\Delta g = g_m(\hat{p}) - g_m(p^*) = \frac{1}{2} \Delta p^\top \nabla^2 g_m(\xi_g) \Delta p,
		\]
		where $\xi_g$ lies between $p^*$ and $\hat{p}$. Consequently,
		\begin{equation*}
			|\Delta g| \le C_m\|\Delta p\|^2 \le C_m\|\Delta\theta\|^2.
		\end{equation*}
		Taking expectations on both sides yields
		\[
		E[\Delta g] \le C_m E[\|\Delta\theta\|^2] \leq C_m N^{-1}, \quad
		E[(\Delta g)^2] \le C_m E[\|\Delta\theta\|^4] \leq C_m N^{-2}.
		\]
		Finally, the proof is completed by the inequality $\Var(\Delta g) \le E[(\Delta g)^2] \leq C_m N^{-2}$.
	\end{proof}
	
Theorem \ref{the:yijie} provides only the convergence rate, without an explicit representation of the leading-order term. We show that such a representation can be obtained under an additional assumption.

	\begin{assumption}\label{assu:2}
	The parameters $\hat{\beta}_n$, $n=0,...,m$, and $\hat{\alpha}$ are estimated independently, i.e., $(\varepsilon_0, \dots, \varepsilon_m, \varepsilon_{m+1})$ is a vector of independent standard normal random variables.
	\end{assumption}
	\begin{theorem}
		\label{the:erjie}
		Under Assumptions \ref{assu:1.1} and \ref{assu:2}, the computational efficiency loss satisfies
            \[
			E(\Delta g) = \frac{\operatorname{tr}(A)}{2N} + O(N^{-3/2}), \quad \Var(\Delta g) = \frac{\operatorname{tr}(A^2)}{2N^2} + o(N^{-2}),
			\] 
			where $A = D (G^\top HG)D$ with $H = \nabla^2 g_m(p^*)$ and $G=\nabla p(\theta)\in \mathbb{R}^{(m+1)\times(m+2)}$. Moreover, it holds that
			\[
			N\Delta g \xrightarrow{d} \frac{1}{2}\,\varepsilon^\top A\varepsilon,
			\]
			where $\varepsilon=(\varepsilon_0, \dots, \varepsilon_m, \varepsilon_{m+1})$.
	\end{theorem}
	\begin{proof}
		A second-order Taylor expansion of $p_n$ at $\theta$ yields
		\[
		p_n(\hat{\theta}) = p_n^* + \underbrace{\nabla p_n(\theta)\Delta\theta}_{\Delta p_n^{(1)}} +\underbrace{\frac12\Delta\theta^\top\nabla^2 p_n(\eta_n)\Delta\theta}_{R_{p_n}},
		\]
		where $\eta_n$ lies between $\theta$ and $\hat{\theta}$. By equation \eqref{eq:deltatheta} and Lemma \ref{lem:erjie}, the linear term is
		\begin{equation*}
			\Delta p_n^{(1)}=\frac{1}{\sqrt{N}}(\nabla p_n(\theta)D)\hat{\varepsilon}
		\end{equation*}
		while the remainder satisfies
		\begin{equation*}
			|R_{pn}|\le C_m\|\Delta\theta\|^2 \le \frac{C_m}{N}\|\hat{\varepsilon}\|^2.
		\end{equation*}
		For notational convenience, we write
		\[
			\Delta p = \Delta p^{(1)} + R_p,
		\]
		where 
		\begin{equation}
			\Delta p^{(1)} =\begin{pmatrix}
				\Delta p_0^{(1)}\\ \vdots \\ \Delta p_m^{(1)}
			\end{pmatrix}=\begin{pmatrix}
				(\nabla p_0(\theta))\\ \vdots \\ (\nabla p_m(\theta))
			\end{pmatrix}\Delta\theta =G\Delta\theta \label{eq:deltap1}
		\end{equation}
		and $R_p:=(R_{p0},R_{p1},...,R_{pm})^\top$. It follows that
		\[
			\|\Delta p^{(1)}\|\le\frac{\|G\|\cdot\|D\|\cdot\|\hat{\varepsilon}\|}{\sqrt{N}}\leq \frac{C_m}{\sqrt{N}}\|\hat{\varepsilon}\|,
		\]
		and
		\[
			\|R_p\|\le \frac{C_m}{N}\|\hat{\varepsilon}\|^2.
		\] 
		Therefore,
		\[
		\|\Delta p\|\le \|\Delta p^{(1)}\| + \|R_p\|
		\le \frac{C_m\|\hat{\varepsilon}\|}{\sqrt{N}} + \frac{C_m}{N}\|\hat{\varepsilon}\|^2.
		\]
		
		Next, we apply a third-order Taylor expansion of $g_m(p)$ at $p^*$. Since $\nabla g_m(p^*)\Delta p=0$ by Lemma \ref{lem:optimal}, we obtain 
		\begin{align}
		g_m(\hat{p}) &= g_m(p^*) + \frac12\Delta p^\top H\Delta p + R_g\nonumber\\
		&=g_m(p^*)+\underbrace{\frac12(\Delta p^{(1)})^\top H\Delta p^{(1)}}_{\mathcal{Q}} + 
		\underbrace{(\Delta p^{(1)})^\top H R_p}_{S} + 
		\underbrace{\frac12 R_p^\top H R_p}_{T}+ R_g,\label{eq:QT}
		\end{align}
		where $H=\nabla^2 g_m(p^*)$ is the deterministic Hessian matrix of $g_m$ at $p^*$, and $R_g$ denotes the remainder. Here, $\mathcal{Q}$ represents the leading-order term, while $S$, $T$, and $R_g$ are higher-order terms. By Lemma \ref{lem:erjie}, 
		\begin{equation}
		|R_g| = \left|\frac16\sum_{i,j,k=0}^{m}\frac{\partial^3 g}{\partial p_i\partial p_j\partial p_k}(\zeta)\,\Delta p_i\Delta p_j\Delta p_k\right|\le C_m\|\Delta p\|^3\le \frac{C_m\|\hat{\varepsilon}\|^3}{N^{3/2}} + \frac{C_m\|\hat{\varepsilon}\|^6}{N^3}\label{eq:Rg},
		\end{equation}
		where $\zeta$ lies between $p^*$ and $\hat{p}$. Moreover, it holds that
		\[
		|S| \le \|\Delta p^{(1)}\|\cdot\|H\|\cdot\|R_p\|
		\le \frac{C_m\|\hat{\varepsilon}\|^3}{N^{3/2}},
		\]
		and
		\[
		|T| \le \frac12\|H\|\cdot\|R_p\|^2
		\le \frac{C_m\|\hat{\varepsilon}\|^4}{N^2}.
		\]
		Since the moments of $\hat{\varepsilon}$ of any fixed order are bounded, it follows by equation (\ref{eq:QT}), together with the estimates of $S$, $T$ and $R_g$, that
		\begin{equation}
		E(\Delta g) = E(\mathcal{Q}) + O(N^{-3/2}).\label{eq:Edeltag}
		\end{equation}
		
		It remains to analyze the leading-order term $\mathcal{Q}$. Substituting equation \eqref{eq:deltap1} into $\mathcal{Q}$, we obtain
		\begin{equation}
				\mathcal{Q} = \frac12(G\Delta\theta)^\top H(G\Delta\theta)
				= \frac{1}{2N} \hat{\varepsilon}^\top D^\top (G^\top HG)D\hat{\varepsilon}
				= \frac{1}{2N} \hat{\varepsilon}^\top A\hat{\varepsilon}.\label{eq:mathQ}
		\end{equation}
		Since $\hat{\beta}_n$, $\hat{\alpha}$ are estimated independently, it holds that $E(\hat{\varepsilon}_i\hat{\varepsilon}_j)=0$ for $i\neq j$.  Moreover, each $\varepsilon_i$ follows a standard normal distribution, so that $E(\varepsilon_i^2)=1$ and $E(\varepsilon^\top A \varepsilon)=\operatorname{tr}(A)$. Therefore,
		\[
		E(\hat{\varepsilon}^\top A \hat{\varepsilon}) = \operatorname{tr}(A)+\sum_{i=0}^{m+1}(E(\hat{\varepsilon}_i^2)-E(\varepsilon_i^2))A_{i,i}.
		\]
The integration-by-parts formula gives
\begin{align*}
E(\hat{\varepsilon}_i^2)-E(\varepsilon_i^2)&=2\int_{B_N}^{+\infty}(B_N^2-x^2)\frac{e^{-x^2/2}}{\sqrt{2\pi}}dx\\
&=2B_N^2\int_{B_N}^{+\infty}\frac{e^{-x^2/2}}{\sqrt{2\pi}}dx-2\int_{B_N}^{+\infty}\frac{e^{-x^2/2}}{\sqrt{2\pi}}dx-\frac{2B_Ne^{-B_N^2/2}}{{\sqrt{2\pi}}}\\
&=O(e^{-B_N^2/2}/B_N)=O(N^{-q} e^{-N^{2q}/2}),
\end{align*} 
where the standard normal tail satisfies
\[
\int_{B_N}^{+\infty}\frac{e^{-x^2/2}}{\sqrt{2\pi}}dx=\frac{e^{-B_N^2/2}}{\sqrt{2\pi}B_N}(1+O(\frac{1}{B_N^2})),
\]
and $B_N\asymp N^q$, $0<q<1/2$, by Assumption \ref{assu:1.1}. Therefore, 
\begin{equation}
E(\hat{\varepsilon}^\top A \hat{\varepsilon}) = \operatorname{tr}(A)+o(N^{-k}), \label{eq:epsilon}
\end{equation}
for any $k>0$. Combining equations \eqref{eq:Edeltag}, \eqref{eq:mathQ} and \eqref{eq:epsilon}, we obtain
\[
E(\Delta g) = \frac{\operatorname{tr}(A)}{2N} + O(N^{-3/2}).
\]
Similarly, using the independence of the components of \(\varepsilon\) and that $A_{i,j}=A_{j,i}$ for all $i,j$, we have
\begin{align*}
E[(\varepsilon^\top A \varepsilon)^2]&=\sum_{i=0}^{m+1} A_{i,i}^2E(\varepsilon_i^4)+\sum_{i\neq j}A_{i,i}A_{j,j}E(\varepsilon_i^2 \varepsilon_j^2)+2\sum_{i\neq j}A_{i,j}^2E(\varepsilon_i^2 \varepsilon_j^2)\\
&=\sum_{i=0}^{m+1}\sum_{j=0}^{m+1} A_{i,i}A_{j,j}+2\sum_{i=0}^{m+1}\sum_{j=0}^{m+1}A_{i,j}^2=(\operatorname{tr}(A))^2+2\operatorname{tr}(A^2),
\end{align*} 
where $E(\varepsilon_i^4)=3$, and 
\[
\Var(\varepsilon^\top A \varepsilon)=E[(\varepsilon^\top A \varepsilon)^2]-(E(\varepsilon^\top A \varepsilon))^2=2\operatorname{tr}(A^2).
\]
Consequently,
\[
\Var(\Delta g) = \frac{\operatorname{tr}(A^2)}{2N^2} + o(N^{-2}).
\]

		Finally, multiplying $\Delta g$ by $N$ yields
		\[
		N\Delta g = \frac12\hat{\varepsilon}^\top A \hat{\varepsilon} + N(S + T + R_g).
		\]
		As $\hat{\varepsilon}^\top A \hat{\varepsilon}$ converges to $\varepsilon^\top A\varepsilon$ in distribution, and $N(S + T + R_g)$ converges to zero in $L^1$ norm, by Slutsky's theorem, we have $N\Delta g \xrightarrow{d} \frac{1}{2}\,\varepsilon^\top A\varepsilon$. This completes the proof.
	\end{proof}
	\begin{remark}
	Although the truncated normal random variable $\hat{\varepsilon}_i$ depends on the parameter $q$, Theorem \ref{the:erjie} shows that this dependence does not affect the convergence rate, since the resulting truncation error is of order $O(N^{-q} e^{-N^{2q}/2})$, which is asymptotically negligible compared with the polynomial convergence terms.
	\end{remark}
	
	\section{Infinite-Dimensional Error Analysis}\label{sect4}
	
	We have analyzed the computational efficiency loss under a fixed truncation $m$. However, the original optimization problem \eqref{eq:gp} is infinite-dimensional. In this section, we extend the error analysis to this setting under mild assumptions on the sequences $\{t_n\}$ and $\{\beta_n\}$.
	
	For a time-discrete scheme with strong convergence order $p>1/2$, the second moment of the increment satisfies $\beta_n = O(2^{-2pn})$, and its tail exhibits an asymptotically geometric decay. Motivated by these properties, we propose the following assumption.
	
	\begin{assumption} The sequences $\{t_n\}$ and $\{\beta_n\}$ satisfy
		\label{assu:3.1}
		\begin{equation*}
			\beta_{n} = C_{\beta}\cdot2^{-2pn}(1+O(2^{-\rho n})),\quad p>1/2, \quad \rho>0, \quad t_{n} = 2^n,
		\end{equation*}
		where $C_{\beta}$ is a constant.
	\end{assumption}
Assumption \ref{assu:3.1} introduces the parameter $\rho$ to characterize the rate of the non-leading-order term. As shown later, the leading-order truncation error term cancels, resulting in a higher-order convergence rate that depends on $\rho$.
	
	Under Assumption \ref{assu:3.1}, \citet{zheng2026optimal} proposed an algorithm to approximate the optimal distribution. Specifically, they approximate it by
\begin{align*}
&\hat{p}_n=\sqrt{\frac{\hat{\beta}_n}{\hat{\alpha}^2+\hat{c}t_n}},\quad n=0,1,2,...,m\\
&\hat{p}_{n+1}=2^{-(2p+1)/2}\cdot\hat{p}_{n},\quad n=m,m+1,...
\end{align*}
where $\hat{c}$ is the solution to the equation
\begin{equation}
\sum_{n=0}^{m}\sqrt{\frac{\hat{\beta}_n}{\hat{\alpha}^2+\hat{c}t_n}}+\frac{2^{-(2p+1)/2}}{1-2^{-(2p+1)/2}}\sqrt{\frac{\hat{\beta}_m}{\hat{\alpha}^2+\hat{c}t_m}}=1.\label{eqPn}
\end{equation}
This algorithm guarantees that $\sum_{n=0}^{\infty}\hat{p}_n=1$.	
	
	Our goal is to estimate 
	\[
	g(\hat{p})-g(p^*),
	\]
	where $g(\hat{p})$ is the computational efficiency with $\hat{p}:=(\hat{p}_0,...,\hat{p}_n,...)$ from the above approximation, and $g(p^*)$ is the optimal computational efficiency. Recall that
\[
g(p^*)=\left(\sum_{n=0}^{\infty}\frac{\beta_n}{p_n^*} - \alpha^2\right)\left(\sum_{n=0}^{\infty} t_n p_n^*\right),
\]
where 
\[
p_n^*=\sqrt{\frac{\beta_n}{\alpha^2+c^*t_n}}
\]
for all $n$, and $c^*$ is the solution to the equation
\[
\sum_{n=0}^{\infty}\sqrt{\frac{\beta_n}{\alpha^2+c^*t_n}}=1.
\]

	For notational convenience, let
	\[
	\gamma = \frac{2^{-(p-1/2)}}{1-2^{-(p-1/2)}}, \quad s =\frac{2^{-(p+1/2)}}{1-2^{-(p+1/2)}}.
	\]
By the definition of $\hat{p}$, we have
	\begin{align*}
		g(\hat{p})
		&=\left(\sum_{n=0}^{m}\frac{\beta_n}{\hat{p}_n}+\sum_{n=m+1}^{\infty}\frac{\beta_n}{\hat{p}_n} - \alpha^2\right)\left(\sum_{n=0}^{m} t_n \hat{p}_n+\sum_{n=m+1}^{\infty} t_n \hat{p}_n\right),\\
		&=\left(\sum_{n=0}^{m}\frac{\beta_n}{\hat{p}_n}+\frac{1}{\hat{p}_m}\sum_{n=m+1}^{\infty} \frac{\beta_n}{2^{-(n-m)(p+1/2)}} - \alpha^2\right)\left(\sum_{n=0}^{m} t_n \hat{p}_n+\gamma\cdot t_m \hat{p}_m\right).\\
	\end{align*}
From equation (\ref{eqPn}), we get
	\[
	\sum_{n=0}^{m-1}\hat{p}_n+(1+s)\hat{p}_m=1.
	\]
However, we observe that
\[
(1+s)\hat{p}_m=(1+s)\sqrt{\frac{\hat{\beta}_m}{\hat{\alpha}^2+\hat{c}t_m}}\neq\sqrt{\frac{(1+\gamma)\hat{\beta}_m}{\hat{\alpha}^2+\hat{c}(1+\gamma)t_m}}.
\]
Therefore, the infinite-dimensional problem cannot be directly transformed into the corresponding finite-dimensional problem.
	
To address this difficulty, we introduce an intermediate probability mass function \(p(m)\) and decompose the efficiency loss as
\begin{equation}
	g(\hat{p})-g(p^*)=(g(\hat{p})-g(p(m)))+(g(p(m))-g(p^*)),\label{eq4.1}
\end{equation}	
where $p(m)$ is defined in an anaglous manner to $\hat{p}$, except that the estimated parameters $\hat{\beta}_n$ and $\hat{\alpha}$ are replaced by their exact values. Specifically,
\begin{align*}
&p_n(m)=\sqrt{\frac{\beta_n}{\alpha^2+c(m)t_n}},\quad n=0,1,2,...,m\\
&p_{n+1}(m)=2^{-(2p+1)/2}\cdot p_{n}(m),\quad n=m,m+1,...
\end{align*}
where $c(m)$ satisfies the equation
\[
\sum_{n=0}^{m}\sqrt{\frac{\beta_n}{\alpha^2+c(m)t_n}}+s\cdot\sqrt{\frac{\beta_m}{\alpha^2+c(m)t_m}}=1.
\]	
As with \(\hat p\), this construction ensures that $\sum_{n=0}^{\infty}p_n(m)=1$.

We observe that the first term $g(\hat{p})-g(p(m)$ in \eqref{eq4.1} contains the error arising from parameter estimation, whereas the second term $g(p(m))-g(p^*)$ is solely due to truncation. We first analyse the latter.

\begin{lemma}\label{lem4.1}
Suppose that Assumption \ref{assu:3.1} is satisfied. Then it holds that
\[
|g(p(m))-g(p^*)|\leq C 2^{-m(p-1/2+2\eta)},
\]
where $\eta=\min\{\rho,1\}$. 
\end{lemma}
\begin{proof}
Let $p=p^*+\lambda(p(m)-p^*)$, where $\lambda\in [0,1]$, and define 
\[
\phi(\lambda)=g(p).
\]
Since $\phi(0)=g(p^*)$ and $\phi(1)=g(p(m))$, we have
\[
g(p(m))-g(p^*)=\phi(1)-\phi(0).
\]
Applying the Taylor expansion of $\phi$ around $\lambda=0$, we obtain
\[
\phi(1)-\phi(0)=\phi^{'}(0)+\phi^{''}(\epsilon)
\]
where $\epsilon\in (0,1)$. 

For the first derivative term $\phi^{'}(0)$, direct calculation gives
\[
\phi^{'}(0)=-\sum_{n=0}^{\infty}\frac{\beta_n}{p_n^{*2}}(p_n(m)-p_n^*)\sum_{n=0}^{\infty} t_n p_n^*+\left(\sum_{n=0}^{\infty}\frac{\beta_n}{p_n^*} - \alpha^2\right)\sum_{n=0}^{\infty}t_n(p_n(m)-p_n^*).
\]
The interchange of differentiation and summation is justified, because the series obtained by term-by-term differentiation of the series defining $\phi(\lambda)$ converges uniformly on $[0,1]$. More precisely, there always exist positive constants $c_m$ and $C_m$, such that both $p_n(m)$ and $p_n^*$ satisfy
\[
c_m 2^{-(p+1/2)n} \leq p_n(m),p_n^*\leq C_m 2^{-(p+1/2)n},
\]
for all $n\geq 0$, hence 
\[
c_m 2^{-(p+1/2)n} \leq p_n\leq C_m 2^{-(p+1/2)n},
\]
for any $p_n$ between $p_n(m)$ and $p_n^*$. Let $p_0=1-\sum_{n\neq 0}p_n$. Since $p^*$ is the minimizer of $g$ under the constraint $\sum_{n=0}^{\infty}p^*_n=1$, the first-order optimality condition implies 
\[
\frac{\partial g(p)}{\partial p_i}|_{p=p^*}=0, \quad i\geq 1.
\] 
Consequently,
\[
-\frac{\beta_i}{p_i^{*2}}\sum_{n=0}^{\infty} t_n p_n^*+\left(\sum_{n=0}^{\infty}\frac{\beta_n}{p_n^*} - \alpha^2\right)t_i=-\frac{\beta_0}{p_0^{*2}}\sum_{n=0}^{\infty} t_n p_n^*+\left(\sum_{n=0}^{\infty}\frac{\beta_n}{p_n^*} - \alpha^2\right)t_0,
\]
for any $i\geq 1$. Therefore,
\[
\phi^{'}(0)=\left(-\frac{\beta_0}{p_0^{*2}}\sum_{n=0}^{\infty} t_n p_n^*+\left(\sum_{n=0}^{\infty}\frac{\beta_n}{p_n^*} - \alpha^2\right)t_0\right)\sum_{i=0}^{\infty}(p_i(m)-p_i^*)=0,
\]
because $\sum_{i=0}^{\infty}p_i(m)=\sum_{i=0}^{\infty}p_i^*=1$. 

For the second derivative term $\phi^{''}(\epsilon)$, we have
\begin{equation}
\phi^{''}(\epsilon)=2\sum_{n=0}^{\infty}\frac{\beta_n}{p_n^{3}}(p_n(m)-p_n^*)^2\sum_{n=0}^{\infty} t_n p_n-2\sum_{n=0}^{\infty}\frac{\beta_n}{p_n^{2}}(p_n(m)-p_n^*)\sum_{n=0}^{\infty}t_n(p_n(m)-p_n^*),\label{phi2}
\end{equation}
where $p=p^*+\epsilon (p(m)-p^*)$. The interchange of differentiation and summation is valid, since the series of second derivatives also converges uniformly on $[0,1]$. The remaining proof consists of estimating the two terms in (\ref{phi2}). Under Assumption \ref{assu:3.1}, we get
\[
q_n(c^*):=\sqrt{\frac{\beta_n}{\alpha^2+c^*t_n}}=C_q\sqrt{ 2^{-(2p+1)n}\frac{1+O(2^{-\rho n})}{1+O(2^{-n})}}=C_q2^{-n(p+1/2)}\left(1+O(2^{-\eta n})\right),
\]
where $\eta=\min\{\rho,1\}$ and $C_q$ is a constant. Define
\[
S_m(c)=\sum_{n=0}^{m}q_n(c)+s\cdot q_m(c),\quad T_m(c)=\sum_{n=m+1}^{\infty}q_n(c),
\]
where $s=\frac{2^{-(p+1/2)}}{1-2^{-(p+1/2)}}$. Then, we have
\[
S_m(c^*(m))-S_m(c^*)=(c^*(m)-c^*)S_m^{'}(\xi),
\]
where $\xi$ is between $c^*$ and $c^*(m)$. Since $S_m(c^*(m))=1$ and $S_m(c^*)=1-T_m(c^*)+s\cdot q_m(c^*)$, we obtain
\begin{equation}
	|\Delta c|=|c^*(m)-c^*|=\left|\frac{T_m(c^*)-s\cdot q_m(c^*)}{S_m^{'}(\xi)}\right|\leq \left|\frac{T_m(c^*)-s\cdot q_m(c^*)}{p_0^{'}(\xi)}\right|\leq C 2^{-m(p+1/2+\eta)},\label{eq4.2}
\end{equation}
where
\[
T_m(c^*)=s q_m(c^*)+O\left(2^{-m(p+1/2+\eta)}\right).
\]

Using this estimate, for $n\leq m$, it holds that
\begin{align}
|p_n(m)-p_n^*|&=\left|\sqrt{\frac{\beta_n}{\alpha^2+c(m)t_n}}-\sqrt{\frac{\beta_n}{\alpha^2+c^*t_n}}\right|\nonumber\\
&=\left|\frac{\sqrt{\beta_n}t_n\Delta c}{(\alpha^2+c(m)t_n)\sqrt{\alpha^2+c^*t_n}+(\alpha^2+c^*t_n)\sqrt{\alpha^2+c(m)t_n}}\right|\nonumber\\
&\leq C\sqrt{\frac{\beta_n}{t_n}}|\Delta c|=O(2^{-(p+1/2)n-(p+1/2+\eta)m}).\label{pnm1}
\end{align}
For $n>m$, since $p_{n+1}(m)=2^{-(2p+1)/2}p_{n}(m)$, we have
\[
p_n(m)=\sqrt{\frac{2^{-2p(n-m)}\beta_m}{2^{n-m}\alpha^2+c(m)t_n}},
\]
and a similar calculation gives
\begin{align}
&|p_n(m)-p_n^*|\nonumber\\
&=\frac{\left|(2^{-2p(n-m)}\beta_m-2^{n-m}\beta_n)\alpha^2+(2^{-2p(n-m)}c^*\beta_m-c(m)\beta_n)t_n\right|}{\sqrt{(2^{n-m}\alpha^2+c(m)t_n)2^{-2p(n-m)}\beta_m}(\alpha^2+c^*t_n)+\sqrt{(\alpha^2+c^*t_n)\beta_n}(2^{n-m}\alpha^2+c(m)t_n)}\label{eq4.3}
\end{align}
Let's focus on the term (\ref{eq4.3}). The denominator of (\ref{eq4.3}) is larger than $C\sqrt{\beta_n t_n^3}$. Furthermore, $2^{n-m}\beta_n>2^{-2p(n-m)}\beta_m\asymp \beta_n$ for sufficiently large $m$, and it follows that
\[
\left|(2^{-2p(n-m)}\beta_m-2^{n-m}\beta_n)\alpha^2\right|\leq C2^{n-m}\beta_n.
\]
Therefore, we have
\begin{align}
&|p_n(m)-p_n^*|\nonumber\\
&\leq C\frac{2^{n-m}\beta_n+\left|c(m)(2^{-2p(n-m)}\beta_m-\beta_n)t_n\right|+\left|(2^{-2p(n-m)}\beta_m \Delta c) t_n\right|}{\sqrt{\beta_n t_n^3}}\nonumber\\
&\leq C\frac{2^{-(2p-1)n-m}+2^{-(2p-1)n-\rho m}+2^{-(2p-1)n-m(p+1/2+\eta)}}{2^{(-p+3/2)n}}\leq C2^{-(p+1/2)n-\eta m},\label{pnm2}
\end{align}
for any $n>m$. 

Moreover, recall that $\Delta c\rightarrow 0$ as $m\rightarrow \infty$. Since $c^*>0$, for any $n\leq m$, we have
\[
p_n(m)=\sqrt{\frac{\beta_n}{\alpha^2+c(m)t_n}}\geq \sqrt{\frac{\beta_n}{\alpha^2+2c^*t_n}}\geq C 2^{-(p+1/2)n},
\]
when $m$ is sufficiently large. For any $n>m$, the tail recursion yields
\[
p_{n}(m)=2^{-(p+1/2)(n-m)}p_m(m)\geq C 2^{-(p+1/2)n}.
\]
Thus, there exists a constant $C_l$, independent of $m$, such that 
\[
p_n\geq C_l 2^{-(p+1/2)n}
\]
for all $n\geq 0$, all sufficiently large $m$, and all $p_n$ between $p_n^*$ and $p_n(m)$. Substituting these estimates into (\ref{phi2}), together with  $\beta_n=O(2^{-2pn})$, we obtain
\begin{align*}
\sum_{n=0}^{\infty}\frac{\beta_n}{p_n^{3}}(p_n(m)-p_n^*)^2&=\sum_{n=0}^{m}\frac{\beta_n}{p_n^{3}}(p_n(m)-p_n^*)^2+\sum_{n=m+1}^{\infty}\frac{\beta_n}{p_n^{3}}(p_n(m)-p_n^*)^2\\
&\leq C2^{-(2p+1+2\eta)m}+C2^{-(p-1/2+2\eta)m}\leq C2^{-(p-1/2+2\eta)m},
\end{align*}
and analogously
\[
\sum_{n=0}^{\infty}\frac{\beta_n}{p_n^{2}}|p_n(m)-p_n^*|\leq C2^{-(p-1/2+\eta)m}, \quad \sum_{n=0}^{\infty}t_n|p_n(m)-p_n^*|\leq C2^{-(p-1/2+\eta)m}.
\]
Consequently,
\[
|\phi^{''}(\epsilon)|\leq C2^{-(p-1/2+2\eta)m}+C2^{-(p-1/2+\eta)m}2^{-(p-1/2+\eta)m}\leq C2^{-(p-1/2+2\eta)m}.
\]
Since
\[
g(p(m))-g(p^*)=\phi^{'}(0)+\phi^{''}(\epsilon),
\]
where $\phi^{'}(0)=0$, the proof is complete.
\end{proof}

\begin{remark}
For a time-discrete scheme, the non-leading-order term typically satisfies $\rho=1$. In this case, Lemma \ref{lem4.1} yields
\[
|g(p(m))-g(p^*)|\leq C 2^{-m(p+3/2)}.
\]
\end{remark}
\begin{remark}
In Lemma \ref{lem4.1}, the convergence rate $-m(p-1/2+2\eta)$ contains the term $2\eta$, because the leading-order term in the expansion of $c(m)-c^*$ is cancelled. However, if we repace the infinite-dimensional objective function $g(p(m))$ with the corresponding
finite-dimensional approximation 
\[
g_m(p(m))=\left(\sum_{n=0}^{m}\frac{\beta_n}{p_n(m)}+\gamma\frac{\beta_m}{p_m(m)}-\alpha^2\right)\left(\sum_{n=0}^{m}t_np_n(m)+\gamma t_mp_m^*(m)\right),
\] 
where the tail is approximated only by its leading-order asymptotic term, we may obtain only $|g_m(p(m))-g(p^*)|\leq C 2^{-m(p-1/2+\eta)}$. This explains why directly approximating the infinite-dimensional problem by a finite truncation may lose one order of accuracy.
\end{remark}

\begin{lemma}\label{lem4.2}
Suppose that Assumptions \ref{assu:1.1} and \ref{assu:3.1} are satisfied, and that $\sigma_n/\beta_n$ is uniformly bounded, i.e., there exists a constant $L>0$, such that $\sigma_n/\beta_n<L$ for all $n\geq 0$. Then there exist an integer $N_0$, and positive constants $c_0, C_0, c_1, C_1$, all independent of $m$ and $N$, such that for all $m\geq 0$ and $N>N_0$, it holds that
\[
c_0 2^{-2pn}\leq \hat{\beta}_n\leq C_0 2^{-2pn}, \quad  0\leq n\leq m,
\]
and
\[
c_1 2^{-(p+1/2)n}\leq \hat{p}_n\leq C_1 2^{-(p+1/2)n}, \quad n\geq 0.
\]
\end{lemma}
\begin{proof}
Assumption \ref{assu:1.1} implies that 
\[
\sup_{0\leq n\leq m}\frac{|\hat{\varepsilon}_n|}{\sqrt{N}}\leq \frac{B_N}{\sqrt{N}}\rightarrow 0
\]
as $N\rightarrow +\infty$. Since
$\sigma_n/\beta_n<L$, we have
\[
1-l_N\leq\hat{\beta}_n/\beta_n\leq 1+l_N, \quad 0\leq n\leq m,
\]
where $l_N\in (0,1)$, and $l_N\rightarrow 0$ as $N\rightarrow +\infty$. Consequently, as $\beta_n\asymp 2^{-2np}$, there exist positive constants $c_0$ and $C_0$, all independent of \(m\), such that
\[
c_0 2^{-2pn}\leq \hat{\beta}_n\leq C_0 2^{-2pn}, \quad 0\leq n\leq m.
\]

Next, we focus on the lower bound for $\hat{p}_n$. Recall that $\hat{p}_n=\sqrt{\frac{\hat{\beta}_n}{\hat{\alpha}^2+\hat{c} t_n}}$ for any $0\leq n\leq m$. If $\hat{c}>0$, it holds that 
\[
1=\sum_{n=0}^{m}\sqrt{\frac{\hat{\beta}_n}{\hat{\alpha}^2+\hat{c} t_n}}+s\sqrt{\frac{\hat{\beta}_m}{\hat{\alpha}^2+\hat{c} t_m}}\leq \sum_{n=0}^{m}\sqrt{\frac{\hat{\beta}_n}{\hat{c} t_n}}+s\sqrt{\frac{\hat{\beta}_m}{\hat{c} t_m}}.
\]
Using the upper bound for $\hat{\beta}_n$, together with $t_n=2^n$, we obtain
\[
\hat{c}\leq \left(\sum_{n=0}^{m}\sqrt{\frac{\hat{\beta}_n}{t_n}}+s\sqrt{\frac{\hat{\beta}_m}{t_m}}\right)^2\leq \frac{C_0}{\left(1-2^{-(p+1/2)}\right)^2}=:c_{up}.
\]
This bound also holds when $\hat {c}\le0$. Since $\hat{\alpha}$ is bounded by Assumption \ref{assu:1.1}, i.e, $\alpha_{low}<\hat{\alpha}<\alpha_{up}$, it follows that
\[
\hat{p}_n\geq\sqrt{\frac{\hat{\beta}_n}{\alpha_{up}^2+c_{up} t_n}}\geq c_1 2^{-(p+1/2)n}, \quad 0\leq n\leq m,
\]
where $c_1$ is independent of $m$ and $N$. For $n>m$, the same lower bound follows from the tail recursion $\hat{p}_{n+1}=2^{-(p+1/2)}\hat{p}_{n}$. 

To establish the upper bound for $\hat{p}_n$, let $c^*$ be the solution to 
\[
\sum_{n=0}^{\infty}\sqrt{\frac{\beta_n}{\alpha^2+c^* t_n}}=1.
\]
Theorem 4 in \citep{rhee2015unbiased} implies that $c^*>0$. Define
\[
\mathcal{S}_m(c):=\sum_{n=0}^{m}\sqrt{\frac{\beta_n}{\alpha^2+c t_n}}, \quad \hat{S}_m(c):=\sum_{n=0}^{m}\sqrt{\frac{\hat{\beta}_n}{\hat{\alpha}^2+c t_n}}+s\sqrt{\frac{\hat{\beta}_m}{\hat{\alpha}^2+c t_m}},
\]
which are monotone decreasing functions of $c>0$. Since
\[
\sum_{n=0}^{\infty}\sqrt{\frac{\beta_n}{\alpha^2+\frac{c^*}{2} t_n}}>1,
\]
there exists a fixed integer $m_0>0$, such that $\mathcal{S}_{m_0}(c^*/2)>1$. Because $l_N\rightarrow 0$, we can increase $N_0$, independent of $m$, so that
\[
\frac{\sqrt{1-l_N}}{1+l_N} \mathcal{S}_{m_0}(c^*/2)>1, \quad N>N_0,
\]
Thus, for \(m\ge m_0\) and \(N\ge N_0\), 
\[
\hat{S}_m(c^*/2)>\frac{\sqrt{1-l_N}}{1+l_N} \mathcal{S}_m(c^*/2)>\frac{\sqrt{1-l_N}}{1+l_N} \mathcal{S}_{m_0}(c^*/2)>1.
\]
Since $\hat{S}_m(\hat{c})=1$, strict monotonicity yields
\[
\hat{c}>c^*/2=:c_{low}.
\]
Therefore, for \(m\ge m_0\) and \(0\le n\le m\), we obtain
\[
\hat{p}_n\leq \sqrt{\frac{\hat{\beta}_n}{\alpha_{low}^2+c_{low} t_n}}\leq C_1 2^{-(p+1/2)n},
\]
The tail recursion extends this bound to $n>m$. It remains to consider \(m<m_0\). Since $\hat{p}_n<1$, we simply choose a sufficiently large constant $C_1$, such that $C_1 2^{-(p+1/2)m_0}>1$. Then, we have $\hat{p}_n<C_1 2^{-(p+1/2)n}$ for all $n\geq 0$. All constants and the threshold $N_0$ are independent of $m$ and $N$. The proof is complete.
\end{proof}

\begin{lemma}\label{lem4.3}
Under the assumptions of Lemma \ref{lem4.2}, we have 
\[
|E(g(\hat{p})-g(p(m)))|\leq C N^{-1},
\]
where $C$ is independent of $m$.
\end{lemma}
\begin{proof}
Following the proof of Lemma \ref{lem4.1}, we let $p=p(m)+\lambda(\hat{p}-p(m))$, where $\lambda\in [0,1]$, and define
\[
\phi(\lambda)=g(p).
\]
Then, we have
\[
g(\hat{p})-g(p(m))=\phi(1)-\phi(0).
\]
Applying the Taylor expansion of $\phi$ at $\lambda=0$, we obtain
\begin{equation}
E(g(\hat{p})-g(p(m)))=E(\nabla g(p(m))\Delta p)+\frac12E(\Delta p^\top H\Delta p)\label{eq4.4}
\end{equation}
where \(H:=\nabla^2 g(\xi_p)\) is the Hessian matrix of \(g\) at a point \(p(\xi_p)\) between \(\hat{p}\) and \(p(m)\), and $\Delta p=\hat{p}-p(m)$. 

We first estimate the second term of (\ref{eq4.4}). By Lemma \ref{lem4.2} and equation (\ref{eq:pgpp}),
	\[
	|H_{i,j}|\le \frac{t_i\beta_j}{p_j^2(\xi_p)}+\frac{t_j\beta_i}{p_i^2(\xi_p)}\leq C2^{i+j}=:B_{i,j}, \quad i\ne j,
	\]
and
	\[
	|H_{i,i}|\leq \frac{2\beta_i\left(\sum_{n=0}^{\infty}t_np_n(\xi_p)\right)}
	{p_i^3(\xi_p)}+\frac{2\beta_i t_i}{p_i^2(\xi_p)}\leq C2^{(p+3/2)i}+C2^{2i}\leq C2^{(p+3/2)i}=:B_{i,i}.
	\]
Consequently, the Cauchy–Schwarz inequality gives
	\begin{align}
		|E(\Delta p^\top H\Delta p)|
		&\leq \sum_{i,j=0}^{\infty}B_{i,j}E|\Delta p_i^\top \Delta p_j|\nonumber\\
		&\leq \sum_{i,j=0}^{\infty}B_{i,j}
		\sqrt{E[(\Delta p_i)^2]E[(\Delta p_j)^2]}.
		\label{Egpm}
	\end{align}

	Next, we estimate \(E[(\Delta p_i)^2]\). For $i\leq m$, the mean value theorem implies that
	\begin{equation}
		\Delta p_i
		=
		\sum_{k=0}^{m}\frac{\partial p_i(\xi)}{\partial \hat{\beta}_k}
		\frac{\sigma_k}{\sqrt{N}}\hat{\varepsilon}_k
		+
		\frac{\partial p_i(\xi)}{\partial \hat{\alpha}}
		\frac{\sigma_{\alpha}}{\sqrt{N}}\hat{\varepsilon}_{m+1},
		\label{deltapi}
	\end{equation}
where $\xi$ is between $\hat{\theta}$ and $\theta$. A similar calculation as in the proof of Lemma \ref{lem:yijie}, together with Lemma \ref{lem4.2}, yields
	\begin{equation}
	\sup_{\hat{\theta}}\left|
	\frac{\partial p_i}{\partial \hat{\alpha}}
	\right|
	\leq C2^{-(p+1/2)i}, \quad 	\sup_{\hat{\theta}}\left|
	\frac{\partial p_i}{\partial \hat{\beta}_i}
	\right|
	\leq C2^{(p-1/2)i}, \quad \sup_{\hat{\theta}}\left|
	\frac{\partial p_i}{\partial \hat{\beta}_k}
	\right|
	\leq
	C2^{-(p+1/2)i}2^{(p-1/2)k},
	\quad k\neq i.	\label{partialest1}
	\end{equation}
	where $C$ is independent of $m$ and $N$. Since \(E(\hat{\varepsilon}_i^2)<E(\varepsilon_i^2)=1\), we obtain
	\begin{align*}
		E(\Delta p_i^2)&\leq\frac{1}{N}\sum_{k=0}^{m}\sup_{\hat{\theta}}\left(\frac{\partial p_i}{\partial \hat{\beta}_k}\right)^2\sigma_k^2 +\frac{2}{N}\sum_{k_1>k_2=0}^m\sup_{\hat{\theta}}\left|\frac{\partial p_i}{\partial \hat{\beta}_{k_1}}\frac{\partial p_i}{\partial \hat{\beta}_{k_2}}\right|\sigma_{k_1}\sigma_{k_2} \\
		&\quad +\frac{2}{N}\sum_{k=0}^{m}\sup_{\hat{\theta}}\left|\frac{\partial p_i}{\partial \hat{\beta}_k}\frac{\partial p_i}{\partial \hat{\alpha}}\right|\sigma_k\sigma_{\alpha} +\frac{1}{N}\sup_{\hat{\theta}}\left(\frac{\partial p_i}{\partial \hat{\alpha}}\right)^2\sigma_{\alpha}^2 .\\
	\end{align*}
As $\sigma_n/\beta_n<L$ and $\beta_n=O(2^{-2pn})$, it holds that $\sigma_n=O(2^{-2pn})$. The above estimates show that each of these four terms is bounded by
\[
\frac{C}{N}2^{-(2p+1)i}.
\]
Hence, for any $i\leq m$, 
	\begin{equation}
		E(\Delta p_i^2)\leq C2^{-(2p+1)i}\frac{1}{N}. \label{Edelta}
	\end{equation}
	For $i>m$, it follows from $\Delta p_i=2^{-(p+1/2)(i-m)}\Delta p_m$ that
    \begin{equation}
	E(\Delta p_i^2)=2^{-(2p+1)(i-m)}E(\Delta p_m^2)\leq C2^{-(2p+1)i}\frac{1}{N}.\label{Edelatm}
    \end{equation}
	Combining inequalities (\ref{Egpm}), (\ref{Edelta}) and (\ref{Edelatm}), together with the bounds on the Hessian entries, yields
	\begin{equation}
	|E(\Delta p^\top H\Delta p)|\leq \frac{C}{N}\sum_{i>j=0}^{\infty}2^{i+j}\sqrt{2^{-(2p+1)i}2^{-(2p+1)j}}+\frac{C}{N}\sum_{i=0}^{\infty}2^{(p+3/2)i} 2^{-(2p+1)i}\leq \frac{C}{N},\label{estsecond}
	\end{equation}
where both series converge because \(p>1/2\).

	We now estimate the first term of (\ref{eq4.4}). Since $p(m)$ is deterministic, we have
\begin{equation}
E(\nabla g(p(m))\Delta p)=\sum_{n=0}^{\infty}\frac{\partial g(p(m))}{\partial p_n}E(\Delta p_n),\label{eq4.5}
\end{equation}
where the interchange of expectation and infinite summation will be justified later. For any $n\leq m$, a second-order Taylor expansion gives
\[
\Delta p_n=\nabla p_n(\theta) \frac{D\hat{\varepsilon}}{\sqrt{N}} +\frac{1}{2}(D\hat{\varepsilon})^\top\nabla^2 p_n(\xi_{\theta}) D\hat{\varepsilon} \frac{1}{N} ,
\]
where $\xi_{\theta}$ is between $\hat{\theta}$ and $\theta$. Since $E(\hat{\varepsilon})=0$, the first-order term has zero expectation, hence
\[
E(\Delta p_n)=E(\mathcal{Q}_n)\frac{1}{N},
\]
where 
\[
\mathcal{Q}_n=\frac{1}{2}(D\hat{\varepsilon})^\top\nabla^2 p_n(\xi_{\theta}) D\hat{\varepsilon}. 
\]
For any $n>m$, 
\[
E(\Delta p_n)=2^{-(p+1/2)(n-m)}E(\Delta p_m).
\]
Therefore, using $|\frac{\partial g(p(m))}{\partial p_n}|\leq C2^n$, we obtain
\begin{align}
|E(\nabla g(p(m))\Delta p)|&=\left|\sum_{n=0}^{m}\frac{\partial g(p(m))}{\partial p_n}E(\mathcal{Q}_n)\frac{1}{N}+E(\mathcal{Q}_m)\sum_{n=m+1}^{\infty}2^{-(p+1/2)(n-m)}\frac{\partial g(p(m))}{\partial p_n}\frac{1}{N}\right|\nonumber\\
&\leq \frac{C}{N} \left(\sum_{n=0}^{m}2^n |E(\mathcal{Q}_n)|+ 2^m |E(\mathcal{Q}_m)|\right).\label{eq4.6}
\end{align}
It remains to estimate $E(\mathcal{Q}_n)$, and
	\[
	|E(\mathcal{Q}_n)|\leq \frac{1}{2}\sum_{i,j=0}^{m+1}\sigma_i \sigma_j \sup_{\hat{\theta}}|\nabla^2_{i,j} p_n|,
	\]
with $\sigma_{m+1}:=\sigma_{\alpha}$. After further straightforward calculation (see Appendix B), we have
\[
\sup_{\hat{\theta}}\left|\frac{\partial^2 p_i}{\partial \hat{\alpha}^2}\right|\leq C2^{-(p+1/2)i}, \quad \sup_{\hat{\theta}}\left|\frac{\partial^2 p_i}{\partial \hat{\beta}_i^2}\right|\leq C2^{(3p-1/2)i}, 
\]
\[
\sup_{\hat{\theta}}\left|\frac{\partial^2 p_i}{\partial \hat{\beta}_j^2}\right|\leq C2^{-(p+1/2)i+(3p-1/2)j}, \quad j\neq i,
\]
and
\[
\sup_{\hat{\theta}}\left|\frac{\partial^2 p_i}{\partial \hat{\beta}_k \partial\hat{\alpha}}\right|\leq C\sup_{\hat{\theta}}\left|\frac{\partial p_i}{\partial \hat{\beta}_k}\right|, \quad \sup_{\hat{\theta}}\left|\frac{\partial^2 p_i}{\partial \hat{\beta}_k \partial\hat{\beta}_j}\right|\leq C2^{(p-1/2)j}\sup_{\hat{\theta}}\left|\frac{\partial p_i}{\partial \hat{\beta}_k}\right|, \quad j\neq i.
\]
Decomposing
\begin{align*}
\sum_{i,j=0}^{m+1}\sigma_i \sigma_j \sup_{\hat{\theta}}|\nabla^2_{i,j} p_n|&=\sum_{i,j=0, i\neq n, j\neq n, i\neq j}^{m}\sigma_i \sigma_j \sup_{\hat{\theta}}\left|\frac{\partial^2 p_n}{\partial \hat{\beta}_i \partial\hat{\beta}_j}\right|+\sum_{i=0, i\neq n}^{m}\sigma_i^2 \sup_{\hat{\theta}}\left|\frac{\partial^2 p_n}{\partial \hat{\beta}_i^2}\right|\\
&\quad+2\sum_{i=0,i\neq n}^{m}\sigma_i \sigma_n \sup_{\hat{\theta}}\left|\frac{\partial^2 p_n}{\partial \hat{\beta}_i \partial\hat{\beta}_n}\right|+\sigma_n^2 \sup_{\hat{\theta}}\left|\frac{\partial^2 p_n}{\partial\hat{\beta}_n^2}\right|\\
&\quad +2\sum_{i=0, i\neq n}^{m}\sigma_i \sigma_{\alpha} \sup_{\hat{\theta}}\left|\frac{\partial^2 p_n}{\partial \hat{\beta}_i \partial\hat{\alpha}}\right|+2\sigma_n \sigma_{\alpha} \sup_{\hat{\theta}}\left|\frac{\partial^2 p_n}{\partial \hat{\beta}_n \partial\hat{\alpha}}\right|+\sigma_{\alpha}^2 \sup_{\hat{\theta}}\left|\frac{\partial^2 p_n}{\partial\hat{\alpha}^2}\right|,
\end{align*}
and using the above derivative estimates, we observe that each term on the right-hand side is bounded by $C2^{-(p+1/2)n}$. Therefore,
\[
|E(\mathcal{Q}_n)|\leq C2^{-(p+1/2)n}.
\]
With this estimate substituted into (\ref{eq4.6}), we get 
\begin{equation}
|E(\nabla g(p(m))\Delta p)|\leq CN^{-1}, \label{eq4.7}
\end{equation}
and the exchange of expectation and infinite summation in (\ref{eq4.5}) is justified by Fubini’s theorem. Combining equation (\ref{eq4.4}) with estimates (\ref{estsecond}) and (\ref{eq4.7}), we complete the proof.
\end{proof}

\begin{theorem}\label{thm4.4}
Suppose that Assumptions \ref{assu:1.1} and \ref{assu:3.1} are satisfied, and that $\sigma_n/\beta_n$ is uniformly bounded. Then there exist constants $C_1$ and $C_2$, independent of the truncation level $m$ and sample size $N$, such that
\[
E(g(\hat{p})-g(p^*))\leq C_1 N^{-1}+ C_2 2^{-m(p-\frac{1}{2}+2\eta)},
\]
$\eta=\min\{\rho,1\}$.
\end{theorem}	
\begin{proof}
Combining the results of Lemmas \ref{lem4.1} and \ref{lem4.3} with equation (\ref{eq4.1}), the proof is complete.
\end{proof}

\begin{remark}
We emphasize that Theorem \ref{thm4.4} does not require that $\varepsilon_n$, $n=0,..,m+1$ in Assumption \ref{assu:1.1} to be independent. However, even if they are independent, the convergence rate cannot be improved in general using the argument in the proof of this theorem.
\end{remark}
	
	\section{Non-uniform Sample Allocation Methods}\label{sect5}
	
	In the error analysis, the same sample size $N$ was used to estimate all parameters $\beta_0,\cdots,\beta_m,\alpha$. However, it is not clear whether a non-uniform allocation of samples can reduce the computational efficiency loss $\Delta g$ under a fixed total computational budget in prior estimation. This problem will be investigated in this section.
	
	In the view of Section \ref{sect4}, the efficiency loss is
	\[
	E(g(\hat{p})-g(p^*))=E(g(\hat{p})-g(p(m)))+(g(p(m))-g(p^*))
	\]
where $g(p(m))-g(p^*)$ only relies on $m$. Thus, our goal is to find the optimal sample size to minimize $|E(g(\hat{p})-g(p(m)))|$ by minimizing its leading-order term. We assume that Assumptions \ref{assu:1.1}, \ref{assu:2} and \ref{assu:3.1} are satisfied. Since $t_n=2^n$ and $\beta_{n+1}\approx \beta_n 2^{-2p}$ for any $n\geq m$, we approximate 
\[
g(p(m))\approx g_m(p(m)):=\left(\sum_{n=0}^{m}\frac{\beta_n}{p_n(m)}+\frac{\gamma\beta_m}{p_m(m)}- \alpha^2\right)\left(\sum_{n=0}^{m} t_n p_n(m)+\gamma t_m p_m(m)\right),
\]
where
\[
p_n(m)=\sqrt{\frac{\beta_n}{\alpha^2+c(m)t_n}},\quad n=0,1,2,...,m,
\]
and $c(m)$ satisfies the equation
\[
\sum_{n=0}^{m}\sqrt{\frac{\beta_n}{\alpha^2+c(m)t_n}}+s\cdot\sqrt{\frac{\beta_m}{\alpha^2+c(m)t_m}}=1.
\]	
Here, $\gamma = \frac{2^{-(p-1/2)}}{1-2^{-(p-1/2)}}$ and $s =\frac{2^{-(p+1/2)}}{1-2^{-(p+1/2)}}$. Analogously, $g(\hat{p})\approx g_m(\hat{p})$. Therefore, 
\[
E(g(\hat{p})-g(p(m)))\approx E(g_m(\hat{p})-g_m(p(m))),
\]
and we focus on the leading-order term of $E(g_m(\hat{p})-g_m(p(m)))$. 

Let $N_i$, $i=0,1,...,m+1$ be the sample size to simulate $\beta_0,\cdots,\beta_m,\alpha$.
Following a similar calculation as in the proof of Theorem \ref{the:erjie} and Lemma \ref{lem4.3}, we get
\begin{align*}
&E(g_m(\hat{p})-g_m(p(m)))\\
&\approx E(\nabla g_m(p(m))\Delta p)+\frac{1}{2}E(\Delta p^\top H \Delta p)\\
&\approx \frac{1}{2}\sum_{n=0}^{m}\frac{\partial g_m(p(m))}{\partial p_n}E(\hat{\varepsilon}^\top (D(N) \nabla^2 p_n(\theta) D(N))\hat{\varepsilon})+\frac{1}{2}E(\hat{\varepsilon}^\top (D(N) (G^\top H G)D(N))\hat{\varepsilon}),
\end{align*}
where 
\[
D(N)=\mathrm{diag}(\sigma_0/\sqrt{N_0},...,\sigma_{m+1}/\sqrt{N_{m+1}})
\] with $\sigma_{m+1}=\sigma_{\alpha}$. Here, 
\[
G=\nabla p(\theta)=(\nabla p_0(\theta),...,\nabla p_{m}(\theta))^\top, \quad H=\nabla^2 g_m(p(m)).
\]
This approximation shares the same leading-order term as $E(g_m(\hat{p})-g_m(p(m)))$. Note that $p(m)$ is not optimal for the $m$-truncated optimization problem, so $E(\nabla g_m(p(m))\Delta p)\neq 0$. Since $\hat{\varepsilon}$ is a vector of independent truncated standard normal random variables with $E(\hat{\varepsilon}_i)=0$ and $E(\hat{\varepsilon}_i^2)\approx 1$, we have
\[
E(\hat{\varepsilon}^\top (D(N) (G^\top H G)D(N))\hat{\varepsilon})\approx \sum_{i=0}^{m+1} \frac{(G^\top HG)_{ii}\cdot \sigma_i^2}{N_i},
\]
and
\begin{align*}
\sum_{n=0}^{m}\frac{\partial g_m(p(m))}{\partial p_n}E(\hat{\varepsilon}^\top (D(N) \nabla^2 p_n D(N))\hat{\varepsilon}\approx \sum_{i=0}^{m+1} \left(\sum_{n=0}^{m}\frac{\partial g_m(p(m))}{\partial p_n}(\nabla^2 p_n)_{ii}\right)\frac{\sigma_i^2}{N_i}.
\end{align*}
Combining these terms gives
\[
E(g_m(\hat{p})-g_m(p(m)))\approx \frac{1}{2}\sum_{i=0}^{m+1} \frac{K_{i}\sigma_i^2}{N_i},
\]
where 
\begin{equation}
K_i=(G^\top HG)_{ii}+\sum_{n=0}^{m}\frac{\partial g_m(p(m))}{\partial p_n}(\nabla^2 p_n)_{ii}. \label{ki}
\end{equation}

	Let $c_i$ be the computational cost to generate a sample of $\beta_i$, $0\leq i\leq m$, and $c_{m+1} $ be the cost to sample from $\alpha$. Typically, $c_i = t_i = 2^i$, $0\leq i\leq m$, and $c_{m+1}=t_M=2^M$ for a sufficiently large $M$. As these samples are taken independently, the total computational cost is $\sum_{i=0}^{m+1} c_i N_i$. Assuming $K_i>0$ for all $i$. The problem of minimizing the expected efficiency loss with the total computational budget constraint $\mathcal{C}$ can be transformed into the following optimization problem:
		\begin{equation*}
			\min_{N_i>0} \sum_{i=0}^{m+1} \frac{K_{i}\sigma_i^2}{N_i} \quad \text{s.t.} \quad \sum_{i=0}^{m+1} c_i N_i = \mathcal{C}.
		\end{equation*}
Its solution satisfies
		\begin{equation}
		N_i \propto \sqrt{\frac{K_{i}\sigma_i^2}{c_i}},\quad 0\leq i\leq m+1.\label{eqoptimal}
		\end{equation}
Theorem \ref{the:5.1} provides the exact expressions for $K_i$. 

\begin{theorem}\label{the:5.1}
Let $K_i$, $0\leq i\leq m+1$, be defined by (\ref{ki}). Define \[
V_m=\sum_{n=0}^m w_n\frac{\beta_n}{p_n}-\alpha^2,
\qquad
T_m=\sum_{n=0}^m w_n t_n p_n, \qquad Q=\sum_{n=0}^m \ell_n \frac{t_n p_n^3}{\beta_n},
\]
where $p_n=p_n(m)$ satisfies $\sum_{n=0}^{m}\ell_n p_n(m)=1$, and
\[
\ell_n=
\begin{cases}
1, & 0\leq n<m,\\
1+s, & n=m,
\end{cases}
\qquad
w_n=
\begin{cases}
1, & 0\leq n<m,\\
1+\gamma, & n=m.
\end{cases}
\]
with $s=\frac{2^{-(p+1/2)}}{1-2^{-(p+1/2)}}$ and $\gamma=\frac{2^{-(p-1/2)}}{1-2^{-(p-1/2)}}$.

Then we have
\begin{align*}
		&(G^\top HG)_{ii}\\
		&=
		\begin{cases}
			\displaystyle
        \frac{w_i p_i^2}{2\beta_i^2}\left[T_m\frac{\beta_i}{p_i^3}-w_i\frac{\beta_i t_i}{p_i^2}+\frac{\ell_i}{Q}\left(\frac{\beta_i}{p_i^2}\sum_{n=0}^m
w_n\frac{t_n^2p_n^3}{\beta_n}-T_m t_i\right)\right],\quad 0\leq i\leq m,\\[18pt]
			\displaystyle
			\frac{2\alpha^2}{Q}\left[\left(\sum_{n=0}^mw_n\frac{p_n^3}{\beta_n}\right)\left(\sum_{n=0}^mw_n \frac{t_n^2 p_n^3}{\beta_n}\right)-\left(\sum_{n=0}^m w_n \frac{t_n p_n^3}{\beta_n}\right)^2\right],\quad i=m+1.
		\end{cases}
\end{align*}
The remaining contribution to $K_i$ is
\[
\sum_{n=0}^{m}\frac{\partial g_m(p(m))}{\partial p_n}(\nabla^2 p_n)_{ii}=\sum_{n=0}^m w_n\left(t_nV_m-\frac{\beta_nT_m}{p_n^2}\right)(\nabla^2 p_n)_{ii},
\]
where the second derivatives of $p_n$ are given by
\begin{align*}
		&(\nabla^2 p_n)_{ii}\\
		&=\begin{cases}
			\displaystyle
         -\delta_{ni}\frac{p_i}{4\beta_i^2}
\left(1+2t_ip_i^2\frac{\partial c(m)}{\partial\beta_i}\right)
+\frac{3t_n^2p_n^5}{4\beta_n^2}\left(\frac{\partial c(m)}{\partial\beta_i}\right)^2
-\frac{t_np_n^3}{2\beta_n}\frac{\partial^2 c(m)}{\partial\beta_i^2}, \quad 0\leq i\leq m,\\[18pt]
			\displaystyle
			-\frac{p_n^3}{\beta_n}
+\frac{3p_n^5}{\beta_n^2}\left(\alpha+\frac{t_n}{2}\frac{\partial c(m)}{\partial\alpha}\right)^2
-\frac{t_n p_n^3}{2\beta_n}\frac{\partial^2c(m)}{\partial\alpha^2},\quad\quad i=m+1.
		\end{cases}
\end{align*}
Here, $\delta_{ni}$ denotes the Kronecker delta, i.e., $\delta_{ni}=1$ if $n=i$ and $\delta_{ni}=0$ otherwise. The first and second derivatives of $c(m)$ are 

\[
\frac{\partial c(m)}{\partial\beta_i}=\frac{\ell_i p_i}{\beta_i Q}, \quad \frac{\partial c(m)}{\partial \alpha}
=-\frac{2\alpha }{Q}\sum_{n=0}^m\ell_n\frac{p_n^3}{\beta_n},
\]
\[
\frac{\partial^2 c(m)}{\partial\beta_i^2}
=
-\frac{\ell_i p_i}{2\beta_i^2Q}
+
\frac{\ell_i^2p_i^2}{\beta_i^2Q^2}
\left[
\frac{3}{2Q}
\sum_{j=0}^m
\ell_j\frac{t_j^2p_j^5}{\beta_j^2}
-\frac{t_ip_i^2}{\beta_i}
\right],
\]
\[
\frac{\partial^2c(m)}{\partial\alpha^2}
=
\frac{2}{Q}
\left[
-\sum_{n=0}^m
\ell_n\frac{p_n^3}{\beta_n}
+
3\alpha^2
\sum_{n=0}^m
\ell_n\frac{p_n^5}{\beta_n^2}
\left(
1-\frac{t_n}{Q}
\sum_{j=0}^m
\ell_j\frac{p_j^3}{\beta_j}
\right)^2
\right].
\]

\end{theorem}
\begin{proof}
The proof is straightforward following some standard calculation (see Appendix A).
\end{proof}

In practice, a small number of samples may be used to obtain preliminary estimates of $\hat{\beta}_n$, $\hat{\alpha}$ in order to calculate $K_i$. Substituting these estimates into the allocation formula (\ref{eqoptimal}) yields our first allocation strategy.

Although Theorem \ref{the:5.1} provides explicit expressions for $K_i$, evaluating them requires preliminary estimation of several parameters, which may be time-consuming. A simpler allocation rule can be obtained from their asymptotic behavior as the level $i$ increases.

\begin{theorem}
		\label{the:4.3}
		Suppose that $\sigma_n\asymp \beta_n$ and $c_n=t_n$. Under Assumptions \ref{assu:1.1}, \ref{assu:2} and \ref{assu:3.1}, the optimal sample size satisfies
			\[
			N_i \asymp \sqrt{w_i} 2^{-(2p+1)i/4}, \quad 0\leq i\leq m. 		
			\]
	\end{theorem}
	\begin{proof}
	Recall that
	\[
	K_i=(G^\top HG)_{ii}+\sum_{n=0}^{m}\frac{\partial g_m(p(m))}{\partial p_n}(\nabla^2 p_n)_{ii}.
	\]
	For the first term of $K_i$, since $\beta_i\asymp 2^{-2pi}$, $t_i=2^i$ and $p_i\asymp 2^{-(p+1/2)i}$, Theorem \ref{the:5.1} gives
	\[
	(G^\top HG)_{ii}\asymp \frac{w_i p_i^2}{\beta_i^2}\left(T_m\frac{\beta_i}{p_i^3}\right)\asymp w_i 2^{(3p+1/2)i}.
	\]
	For the second term of $K_i$, we rewrite
	\begin{align*}
	\sum_{n=0}^{m}\frac{\partial g_m(p(m))}{\partial p_n}(\nabla^2 p_n)_{ii}&=-w_i\left(t_i V_m-\frac{\beta_i T_m}{p_i^2}\right)\frac{p_i}{4\beta_i^2}
\left(1+2t_ip_i^2\frac{\partial c(m)}{\partial\beta_i}\right)\\
&\quad+\sum_{n=0}^{m}w_n\left(t_nV_m-\frac{\beta_nT_m}{p_n^2}\right)\frac{3t_n^2p_n^5}{4\beta_n^2}\left(\frac{\partial c(m)}{\partial\beta_i}\right)^2\\
&\quad-\sum_{n=0}^{m}w_n\left(t_nV_m-\frac{\beta_nT_m}{p_n^2}\right)\frac{t_np_n^3}{2\beta_n}\frac{\partial^2 c(m)}{\partial\beta_i^2}.
	\end{align*}
     Similar calculation shows that
	\[
	\frac{\partial c(m)}{\partial\beta_i}\asymp\ell_i 2^{(p-1/2)i}, \quad \frac{\partial^2 c(m)}{\partial\beta_i^2}\asymp\ell_i 2^{(3p-1/2)i}.
	\]
	Moreover, as $\beta_i/p_i^2=\alpha^2+c(m)t_i$ and $\sum_{n=0}^{m}\ell_n p_n=1$, we get
	\begin{align*}
	t_i V_m-\frac{\beta_i T_m}{p_i^2}&=t_i(V_m-c(m)T_m)-\alpha^2 T_m\\
	&=t_i \alpha^2 \left(\sum_{n=0}^{m} w_n p_n-1\right)-\alpha^2 T_m\\
	&=t_i \alpha^2 (\gamma-s)p_m-\alpha^2 T_m.
	\end{align*}
	For $i\le m$, $t_ip_m\le t_mp_m\leq C$. Together with the uniform boundedness of $T_m$, this yields
	\[
	\left|t_i V_m-\frac{\beta_i T_m}{p_i^2}\right|<C.
	\] 
	Substituting these bounds into the second term of $K_i$ gives
	\[
	\sum_{n=0}^{m}\frac{\partial g_m(p(m))}{\partial p_n}(\nabla^2 p_n)_{ii}=O( 2^{(3p-1/2)i}).
	\]
	Consequently, 
	\[
	K_i\asymp w_i 2^{(3p+1/2)i},
	\]
	and therefore,
	\[
	N_i \propto \sqrt{\frac{K_{i}\sigma_i^2}{c_i}}\asymp\sqrt{w_i}2^{-(2p+1)i/4},
	\]
	which completes the proof. 
	\end{proof}
	
	\begin{remark}
	The proof of Theorem \ref{the:4.3} reveals that $\sum_{n=0}^{m}\frac{\partial g_m(p(m))}{\partial p_n}(\nabla^2 p_n)_{ii}=o((G^\top HG)_{ii})$ as $i\rightarrow \infty$, i.e., $K_i \approx (G^\top HG)_{ii}$ when $i$ is sufficiently large.
	\end{remark}
	\begin{remark}
Theorem \ref{the:4.3} suggests that we can choose the sample size according to the following proportions	
\[
N_0:N_1:...:N_m=1:2^{-(2p+1)/4}:...:\sqrt{1+\gamma} 2^{-(2p+1)m/4}.
\]
The parameter $\alpha$ requires separate treatment, because the expression of $K_{m+1}$ differs from those of $K_i$, $0\leq i\leq m$.
	\end{remark}
	
	Motivated by Theorem \ref{the:4.3}, we propose the second allocation rule. We allocate all the computational resources among $N_0$,...,$N_m$ according to 
	\[
	N_i \propto \sqrt{w_i} 2^{-(2p+1)i/4},
	\]
	and estimate $\alpha$ using
	\[
	\alpha=E(Y)\approx\sum_{n=0}^{m}E(\Delta_n)
	\]
	where $E(\Delta_n)$ can be estimated simultaneously with $\beta_n$ using the same $N_n$ samples. Therefore, no additional computational resources need to be allocated to $N_\alpha$. Although the estimator of $\alpha$ is correlated with that of $\beta_n$, the correlation may be weak for sufficiently large $n$. 
	
  Finally, we propose a third approach, in which $\Delta_0,\Delta_1,...,\Delta_m$ are constructed using the same Brownian path. Specifically, we first generate $2^m$ Brownian increments to simulate $Y_m$. We then construct $Y_{m-1}$ by combining consecutive pairs of increments, with the $j$th increment at level $m-1$ obtained as the sum of the $(2j-1)$th and $2j$th increments at level $m$. Repeating this procedure recursively, we obtain $Y_{m-2},\ldots,Y_0$. As a consequence of the coupling across levels, the estimators of $\beta_0,\ldots,\beta_m$ and $\alpha$ are generally dependent.

  The computational efficiency of these sample allocation methods will be compared numerically in the next section.
	
	\section{Numerical Experiments}
    In this section, we conduct numerical experiments using MATLAB to verify the convergence rates derived in Section \ref{sect4} and to examine the efficiency of the non-uniform sample allocation methods proposed in Section \ref{sect5}. Our experiments are based on three stochastic differential equations: the Black-Scholes model, the Heston model and the CIR model. These models are widely applied in finance. 
	
	Recall that the computational efficiency loss is defined as
	\begin{equation*}
		\begin{split}
			\Delta g :&= g(\hat{p}) - g(p^*)\\
			\quad &=\left(\sum_{n=0}^{\infty}\frac{\beta_n}{\hat{p}_n} - \alpha^2\right)\left(\sum_{n=0}^{\infty}t_n\hat{p}_n\right)-\left(\sum_{n=0}^{\infty}\frac{\beta_n}{p_n^*} - \alpha^2\right)\left(\sum_{n=0}^{\infty} t_n p_n^*\right).
		\end{split}
	\end{equation*} 
As the parameters \(\beta_n\) are unknown,  we use at least \(\mathcal{M} = 10^7\) samples to construct reference values for the exact values of parameters \(\beta_n\) as follows:
	\[
	\beta_n \approx \beta_n^{\mathrm{ref}}= 
	\begin{cases}
		\displaystyle \frac{1}{\mathcal{M}}\sum_{i=1}^{\mathcal{M}} \Delta_{n,i}^2, & n = 0,1,\dots,14,\\[10pt]
		\dfrac{\beta_{14}^{\mathrm{ref}}}{4^{p(n-14)}}, & n \ge 14,
	\end{cases}
	\]
	where \(\Delta_{n,i}\) is the increment at level \(n\) in the \(i\)-th simulation, and \(p>1/2\) is the strong convergence rate of the time-discrete scheme. Thus, the reference values of the tail parameters follow the geometric decay
	\[
	\beta_{n+k}^{\mathrm{ref}} = \beta_n^{\mathrm{ref}} / (4^p)^k, \quad n\geq 14.
	\]
	Based on these reference parameters, we compute the reference value \(g_{\mathrm{ref}}\) of the objective function \(g(p^*)\), i.e.,
	\[ 
	g(p^*) \approx g_{\mathrm{ref}} := \left(\sum_{n=0}^{14}\frac{\beta_n^{\mathrm{ref}}}{p_n^{\mathrm{ref}}}+\frac{\beta_{14}^{\mathrm{ref}}}{p_{14}^{\mathrm{ref}}}\cdot\gamma - (\alpha^{\mathrm{ref}})^2\right)\left(\sum_{n=0}^{14}t_np_n^{\mathrm{ref}}+t_{14}p_{14}^{\mathrm{ref}}\cdot\gamma\right).
	\]
	where 
	\[
	p_n^{\mathrm{ref}}=\sqrt{\frac{\beta_n^{\mathrm{ref}}}{\alpha^2+c_{\mathrm{ref}}t_n}}, \qquad n=0,1,\dots,14,
	\]
	in which $c_{\mathrm{ref}}$ satisfies the equation
	\[
	\sum_{n=0}^{14}p_n^{\mathrm{ref}}+s\cdot p_{14}^{\mathrm{ref}}=1.
	\]	
    Here, the exact values of $\alpha=E(Y)$ for all models considered can be computed with very high accuracy without Monte Carlo estimation.
    
	After establishing the reference parameters, we design numerical experiments for each model, aiming to examine the asymptotic orders of $\Delta g$ with respect to the truncation level \(m\) and the sample size \(N\), respectively. The detailed settings are as follows.
	
	\paragraph{Convergence with respect to \(m\).}
	We investigate the convergence rate with respect to the truncation level \(m\). Specifically, let
	\[
	\Delta g_{\mathrm{det}}(m)=|g(p(m))-g(p^*)|.
	\]
	We plot $\log_2(\Delta g_{\mathrm{det}}(m))$ against $m$, and the slope of the resulting curve represents the convergence rate. Theoretically, we have $\Delta g_{\mathrm{det}}(m)=O(2^{-m(p-1/2+2\eta)})$, where $\eta=\min(1,\rho)$, and $\rho$ denotes the non-leading order of $\beta_n$, i.e.,
	\[
	\beta_n=C_\beta 2^{-2pn}\bigl(1+O(2^{-\rho n})\bigr).
	\]
	To examinate whether $\rho=1$, we plot $\log_2 |\beta_n-4^{p}\beta_{n+1}|$ against $n$ and use the slope of the resulting curve to estimate $\rho$, where the slope is $-(2p+\rho)$.

	\paragraph{Convergence with respect to \(N\).}
	For each fixed truncation level \(m\), we illustrate the convergence rate of 
	\[
	|E(g(\hat{p})-g(p(m)))|
	\]
	with respect to the sample size $N$. Specifically, for each fixed $N$ and $m$, we obtain independent estimates \(\hat{\alpha}\) and \(\hat\beta_{n},n=0,1,..,m\), using \(N\) samples at each level, and then calculate the estimated probability mass function $\hat p_N=(\hat p_{0,N},\ldots,\hat p_{m,N})$. Recall that Assumption \ref{assu:1.1} requires that $E(\hat{\alpha})=E(\alpha)$. For the Black-Scholes model and the CIR model, we simulate the terminal values of the stochastic processes exactly, and for the Heston model, we employ a sufficiently small time step of $2^{-10}$, to make the discretization bias negligible.
	
	For each pair of \(m\) and \(N\), the expectation is then approximated by
	\[
	E(g(\hat{p})-g(p(m)))\approx \frac{1}{N_{\mathrm{exp}}}\sum_{j=1}^{N_{\mathrm{exp}}}\Delta g_N^{(j)},
	\]
where $\Delta g_N=g(\hat{p})-g(p(m))$, and \(N_{\mathrm{exp}}\) is the number of ‌repetitions. A small $N_{\mathrm{exp}}=1000$ is typically sufficient to yield estimated Monte Carlo standard errors that are significantly smaller than the absolute values of the corresponding sample means. We verify whether \(|E(g(\hat{p})-g(p(m)))=O(N^{-1})|\), as predicted by the theoretical analysis.
	
	\paragraph{Sample allocation methods.}
	We compare five sample allocation methods with a fixed total computational cost $C$, which are proposed in Section \ref{sect5}. The computational cost to draw a sample of $\Delta_n$ is set to be $c_n=2^n$, and we fix the discretization level used to estimate \(\alpha\) at \(M=7\).
	
\begin{enumerate}
  \item Uniform-allocation method ($A_1$): the same number of samples is assigned to estimate $\beta_0,...,\beta_m$ and $\alpha$, i.e.,
	\[
	N_i^{(A_1)}=N_\alpha^{(A_1)}=\left\lfloor\frac{C}{\sum_{k=0}^{m}c_k+c_M}\right\rfloor,\qquad i=0,\ldots,m,
	\]
	and these parameters are estimated independently. 
  \item Exact-allocation method ($A_2$): the number of samples is assigned according to
    \[
	N_i^{(A_2)}\propto\sqrt{\frac{K_i\sigma_i^2}{c_i}},\qquad i\in\{0,\ldots,m,\alpha\},
	\]
	where $K_i$ is computed using the formula in Theorem \ref{the:5.1} with the reference parameter values. 
  \item Exact-allocation-beta method ($A_{2,\beta}$): the number of samples is assigned to estimate $\beta_0,...,\beta_m$ according to
    \[
	N_i^{(A_{2,\beta})}\propto\sqrt{\frac{K_i\sigma_i^2}{c_i}},\qquad i\in\{0,\ldots,m \},
	\]
	where $K_i$ is from Theorem \ref{the:5.1}. The same increment samples are reused to estimate
	\[
	\alpha\approx\sum_{n=0}^{m}E(\Delta_n)\approx\sum_{n=0}^{m}\frac{1}{N_n}\sum_{j_n=1}^{N_n}\Delta_{n}^{(j_n)}.
	\]
	Hence, no additional computational cost is required to estimate $\alpha$.
  \item Asymptotic-allocation method ($A_3$): the number of samples is assigned to estimate $\beta_0,...,\beta_m$ based on Theorem \ref{the:4.3}, i.e.,
   \[
	N_i^{(A_3)}\propto\sqrt{w_i}\,2^{-(2p+1)i/4},\qquad i=0,\ldots,m,
	\]
	where $w_i=1+\gamma$, if $i=m$, and $w_i=1$, otherwise. The estimate of $\alpha$ is constructed by reusing the increment samples as in $A_{2,\beta}$.
  \item Identical-path method ($A_4$): the same number of samples is assigned to estimate $\beta_0,...,\beta_m$, but $\Delta_0,\Delta_1,...,\Delta_m$ are constructed using the same Brownian path. The estimate of $\alpha$ is constructed by reusing the increment samples.
\end{enumerate}
	
	Moreover, we consider the method ($A_0$) using the subcanonical probability mass function 
	\[
	\hat{p}_n=(1-2^{-(2p+1)/2})2^{-n(2p+1)/2},
	\]
	without any prior estimation of $\beta_0,..\beta_m$ or $\alpha$. This method preserves the optimal decay rate of the tail distribution, while maintaining the property that $\sum_{n=0}^{\infty}\hat{p}_n=1$, and is therefore useful in practice.
	
	We examinate the efficiency of these sample allocation methods by plotting their corresponding $E(\Delta g)$ against the total computational cost $C$, for a fixed \(m\). To estimate $E(\Delta g)$ from each method, $1000$ repetitions of $\Delta g$ are used, which are sufficient to produce an accurate estimate with a significantly small Monte Carlo standard deviation. 
	
	\subsection*{Black-Scholes model}
	
	The Black--Scholes model \citep{black2019pricing} is of the form
	\[
	dS(t)=rS(t)\,dt+\sigma S(t)\,dB(t),
	\]
	where we set
	\[
	S(0)=1,\quad K=1,\quad r=0.05,\quad T=1,\quad \sigma=0.25.
	\]
	The path of $S$ is approximated by the Milstein scheme
	\[
		S_h((j+1)h)=S_h(jh)+rS_h(jh)h
		+\sigma S_h(jh)\Delta B_j+\frac{1}{2}\sigma^2S_h(jh)
		\left((\Delta B_j)^2-h\right),
	\]
	where $\Delta B_j = B((j+1)h)-B(jh)$, \(j=1,2,\ldots,T/h-1\) and \(S_h(0)=S(0)\). We use it to compute the European call option price $E(Y)$, with $Y=e^{-rT}(S(T)-K)^+$.
	
	We first examine the relevant convergence rates, as illustrated in Figure \ref{fig1}. The top-left panel shows the convergence behavior of $\beta_n$ (solid blue line) and its non-leading order terms (solid red line). The dashed lines indicate the corresponding theoretical slopes. We observe that $\beta_n$ converges with order $2$, corresponding to $p=1$, and the non-leading order of $\beta_n$ is $\rho=1$, since the slope of the corresponding line is $2p+\rho=3$. Furthermore, our theoretical analysis predicts that $\Delta g_{\mathrm{det}}(m)=O(2^{-m(p-1/2+2\eta)})$, with $\eta=\min(1,\rho)$, and that $|E(\Delta g_N)|=O(N^{-1})$. The top-right panel demonstrates that $\Delta g_{\mathrm{det}}(m)$ converges with order $2.5$, and the bottom panel demonstrates that $|E(\Delta g_N)|$ converges with order $1$ for each $m=3,7,14$. Interestingly, the curves corresponding to $m=7$ and $m=14$ nearly coincide, suggesting that increasing $m$ beyond $7$ has a negligible effect on this error component. Overall, the observed convergence rates are consistent with the theoretical predictions.
	
		\begin{figure}
		\centering
		\includegraphics[width=0.45\textwidth]{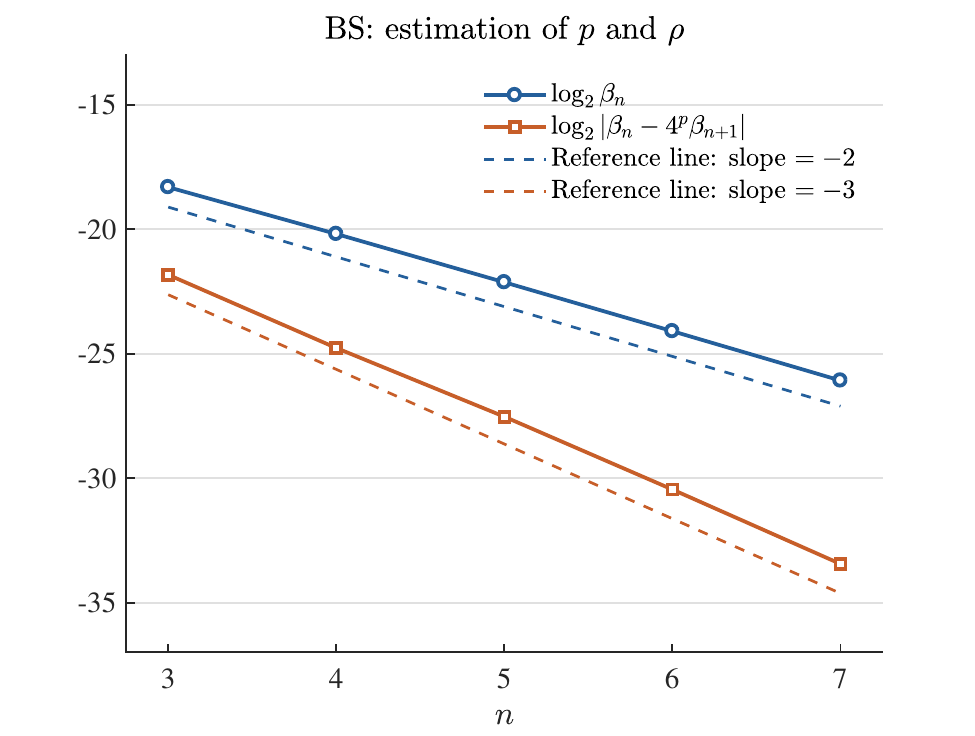}\hfill
		\includegraphics[width=0.45\textwidth]{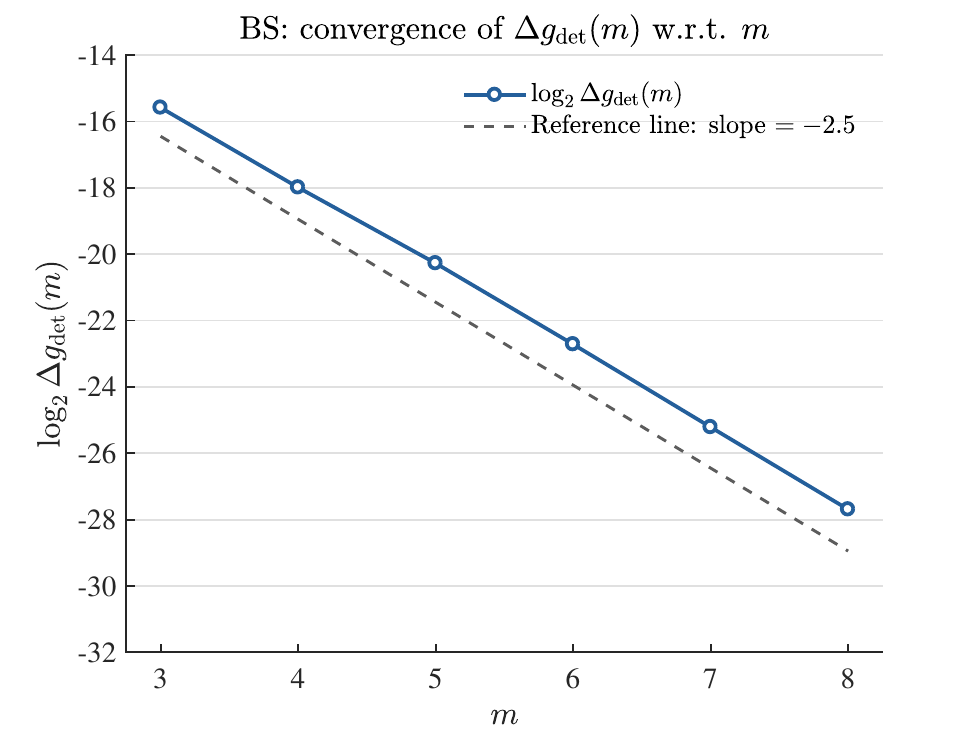}\\[0.5em]
		\includegraphics[width=0.45\textwidth]{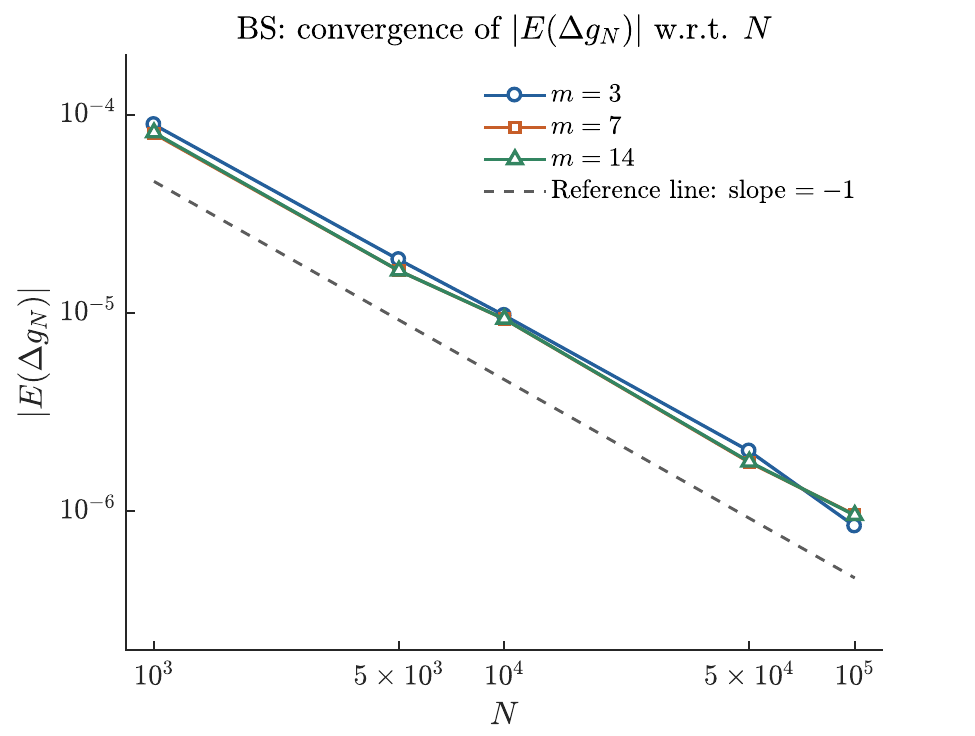}
		\caption{Convergence rates for the Black-Scholes model.}\label{fig1}
	\end{figure}
	
	We then compare the efficiency of five sample allocation methods, with the results presented in Figure \ref{fig4}. All five methods outperform the subcanonical method $A_0$, which requires no prior estimation. Furthermore, for each method, $E(\Delta g)$ approaches the deterministic truncation error $\Delta g_{\mathrm{det}}(m)$ as the total computational cost increases. Among the five methods, all non-uniform allocation strategies outperform the uniform allocation method $A_1$. More specifically, when $m=3$, the Identical-path method $A_4$ achieves the highest efficiency, followed by the Exact-allocation-beta method $A_{2,\beta}$ and the Asymptotic-allocation method $A_3$. When $m=7$, however, the latter two methods exhibit the best performance. Overall, method $A_3$ appears to be the most attractive choice in practice, owing to its simple allocation formula and ease of implementation.
	
		\begin{figure}
		\centering
		\includegraphics[width=0.5\textwidth]{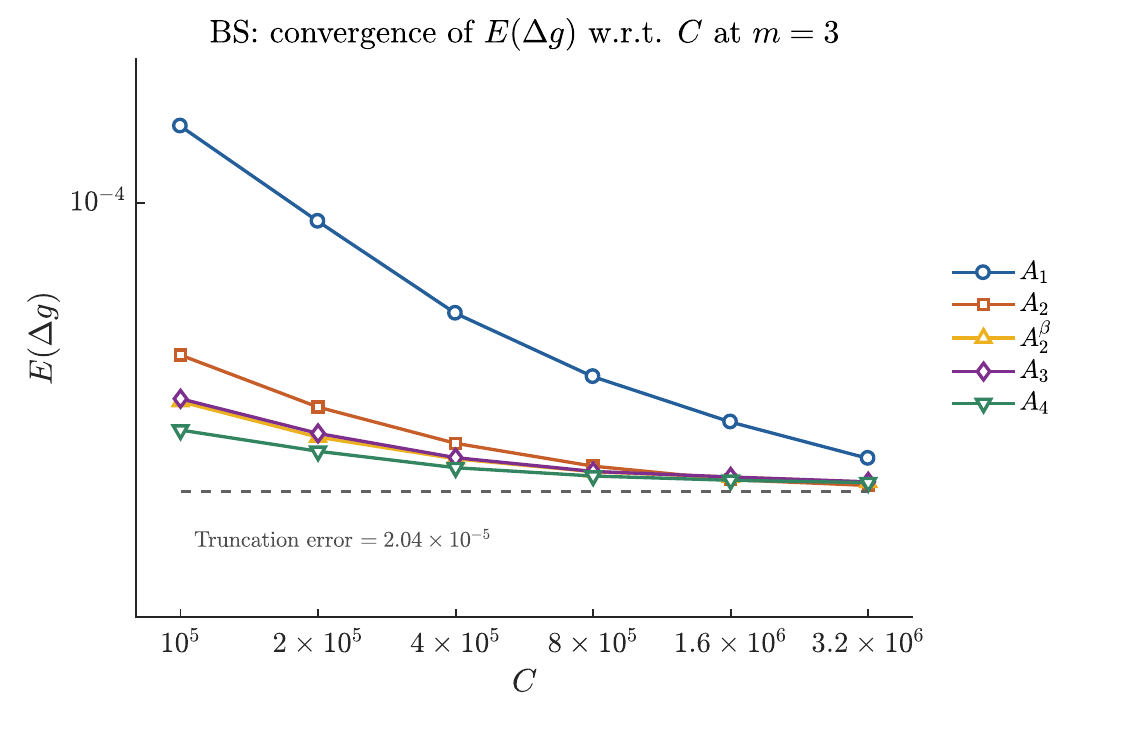}\hfill
		\includegraphics[width=0.5\textwidth]{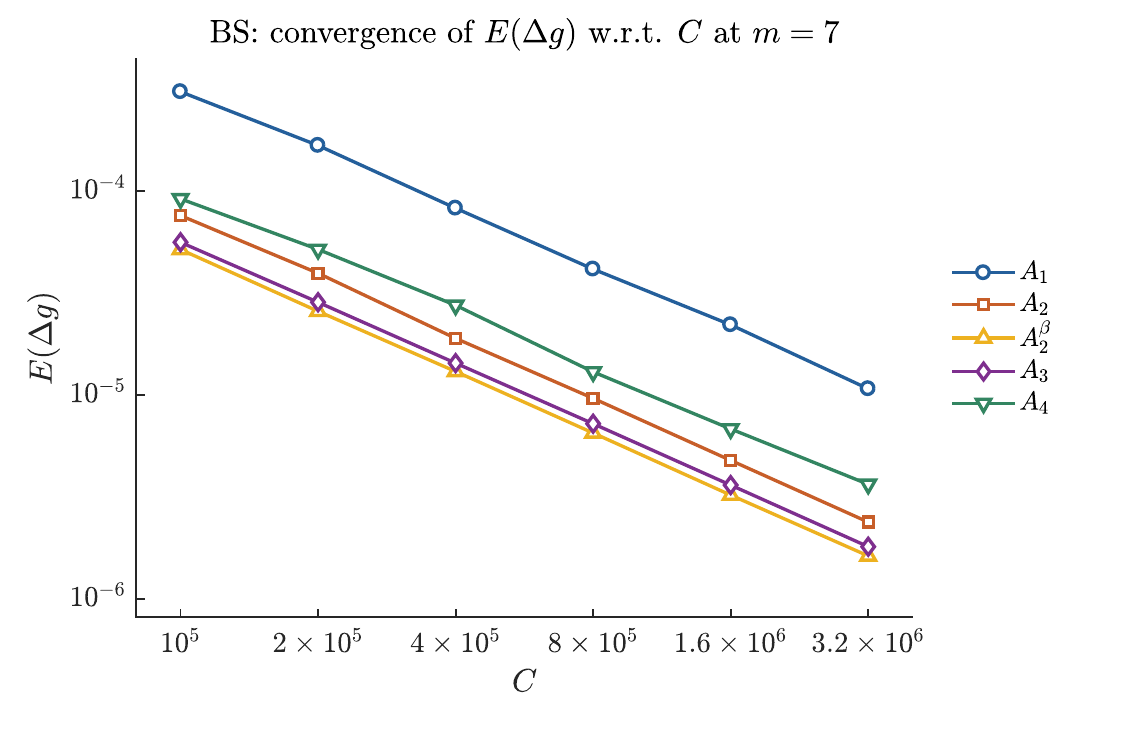}\\[0.5em]
		\caption{Comparison of sample allocation methods for the Black-Scholes model. $E(\Delta g)\approx 7.855\times 10^{-2}$ for method $A_0$; Truncation error $\approx 2.57\times 10^{-8}$ when $m=7$.}\label{fig4}
	\end{figure}
	
	It may be a bit surprising that the Exact-allocation method $A_{2}$ is less efficient than methods $A_{2,\beta}$ and $A_3$. One possible explanation is that $A_{2}$ allocates a relatively large number of samples to estimating $\alpha$, whose estimator exhibits a relatively large variance. Reusing the samples generated for estimating $\beta_n$ to estimate $\alpha$ may therefore lead to considerable computational savings. Nevertheless, when the same sampling strategy is employed for estimating $\alpha$, method $A_{2,\beta}$, based on the exact allocation formula, is slightly more efficient than method $A_3$, which uses an asymptotic allocation formula.

	\subsection*{Heston model}
	The Heston stochastic volatility model \citep{heston1993closed} is specified as
	\[
	\begin{aligned}
		dS(t) &= rS(t)\,dt+\sqrt{V(t)}S(t)\,dB_1(t),\\
		dV(t) &= \kappa(\theta-V(t))\,dt+\sigma\sqrt{V(t)}\,dB_2(t),
	\end{aligned}
	\]
	where \(B_1(t)\) and \(B_2(t)\) are mutually independent standard Brownian motions, and we set
	\[
	\begin{aligned}
		S(0) &= 1,     & T      &= 1,    & K      &= 1,     & r      &= 0.05,\\
		V(0) &= 0.04,  & \kappa &= 1,    & \theta &= 0.04,  & \sigma &= 0.2.
	\end{aligned}
	\]
	Our objective is to estimate $E(Y)$, where $Y=e^{-rT}(S(T)-K)^+$.
	
	The level increments \(\Delta_n\) are constructed using exact simulation of the variance process $V$ at the time-grid points, combined with the stochastic trapezoidal discretization considered in \citep{zheng2023multilevel}. Let
	\[
	I_h:=\sum_{j=0}^{T/h-1}\frac{V(jh)+V((j+1)h)}{2}h.
	\]
	Since \(B_1(t)\) and \(B_2(t)\) are independent, the terminal log-price is approximated by
	\[
	\ln S_h(T)=\ln S(0)+rT-\frac12I_h+\sqrt{I_h}\,Z,
	\qquad Z\sim N(0,1).
	\]
	The fine and coarse approximations are coupled using the same simulated variance path and the same standard normal random variable \(Z\).

   The convergence results for the Heston model are presented in Figure \ref{fig2}. The numerical results indicate that $\beta_n$ converges with order $2$, and the non-leading order of $\beta_n$ appears to be $\rho=1$, despite some fluctuations when $n$ is large. These fluctuations may be attributed to the small coefficient of the higher-order term in the asymptotic expansion of $\beta_n$, which makes an accurate estimation of $\rho$ particularly challenging. Indeed, even with $5\times 10^8$ samples used to estimate each $\beta_n$, the higher-order behavior remains difficult to resolve numerically.  Furthermore, $\Delta g_{\mathrm{det}}(m)$ converges with order $2.5$, and $|E(\Delta g_N)|$ converges with order $1$, consistent with the theoretical ones. 
   
   	\begin{figure}
		\centering
		\includegraphics[width=0.45\textwidth]{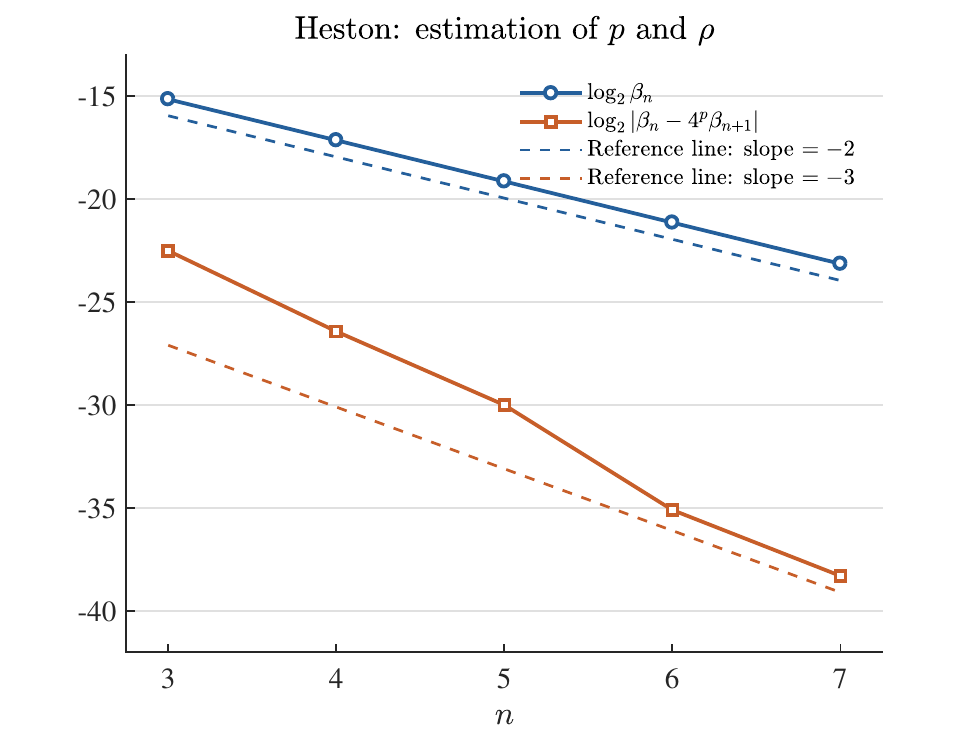}\hfill
		\includegraphics[width=0.45\textwidth]{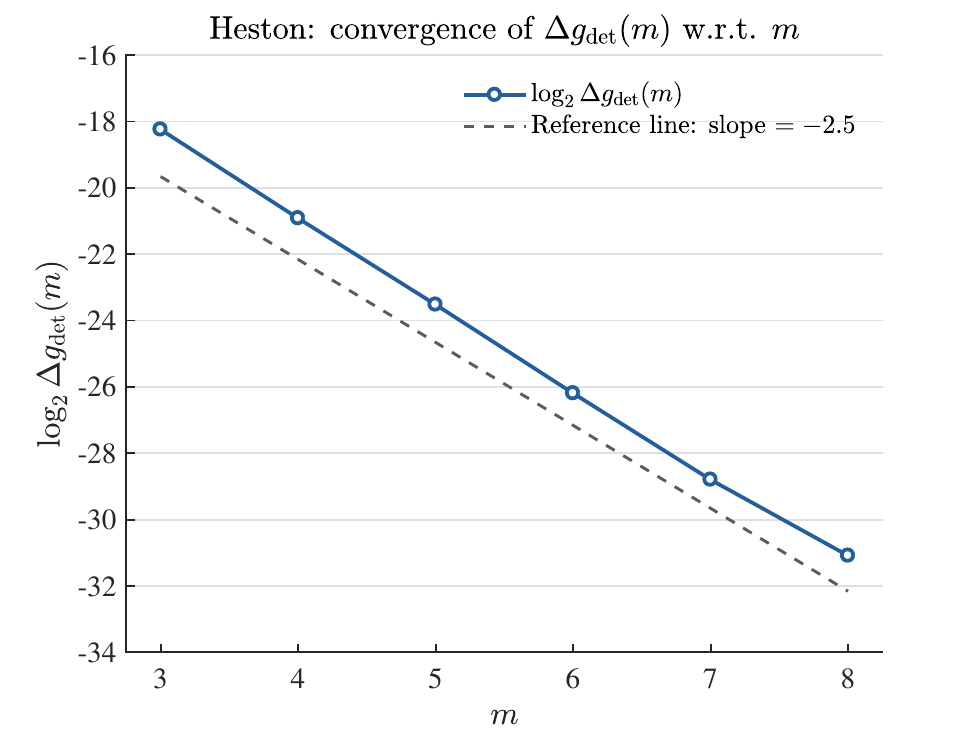}\\[0.5em]
		\includegraphics[width=0.45\textwidth]{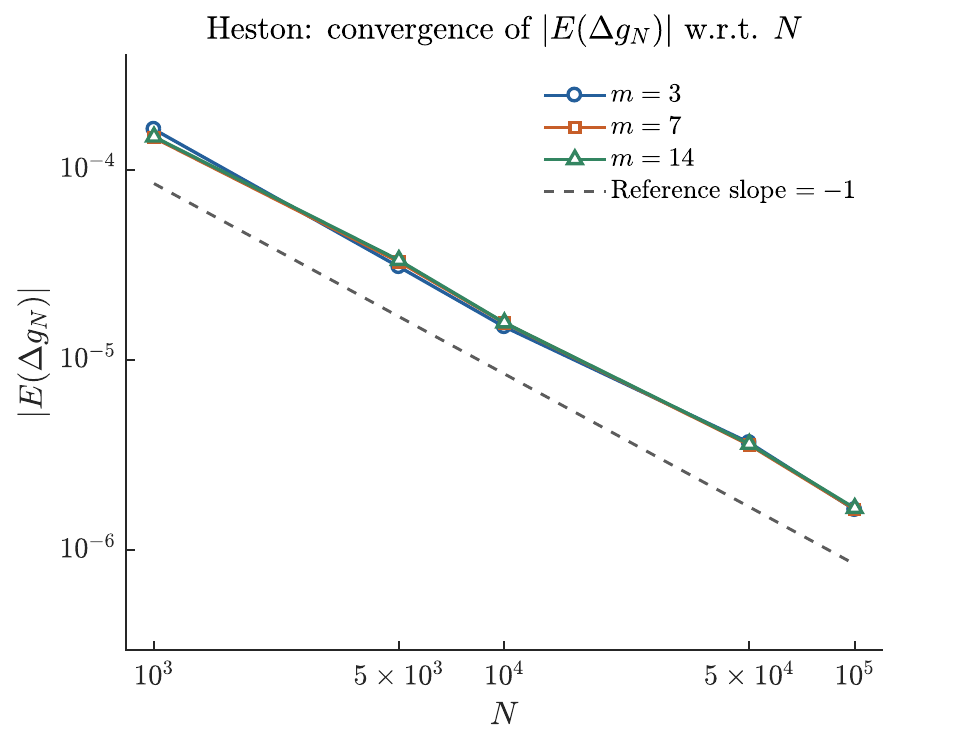}
		\caption{Convergence rates for the Heston model.}\label{fig2}
	\end{figure}
   
	The efficiency comparison of the different sample allocation methods is presented in Figure \ref{fig5}. The results are broadly consistent with those obtained for the Black–Scholes model. In particular, method $A_4$ achieves the highest efficiency when $m=3$, whereas methods $A_{2,\beta}$ and $A_3$ exhibit the best performance when $m=7$.

	\begin{figure}
		\centering
		\includegraphics[width=0.5\textwidth]{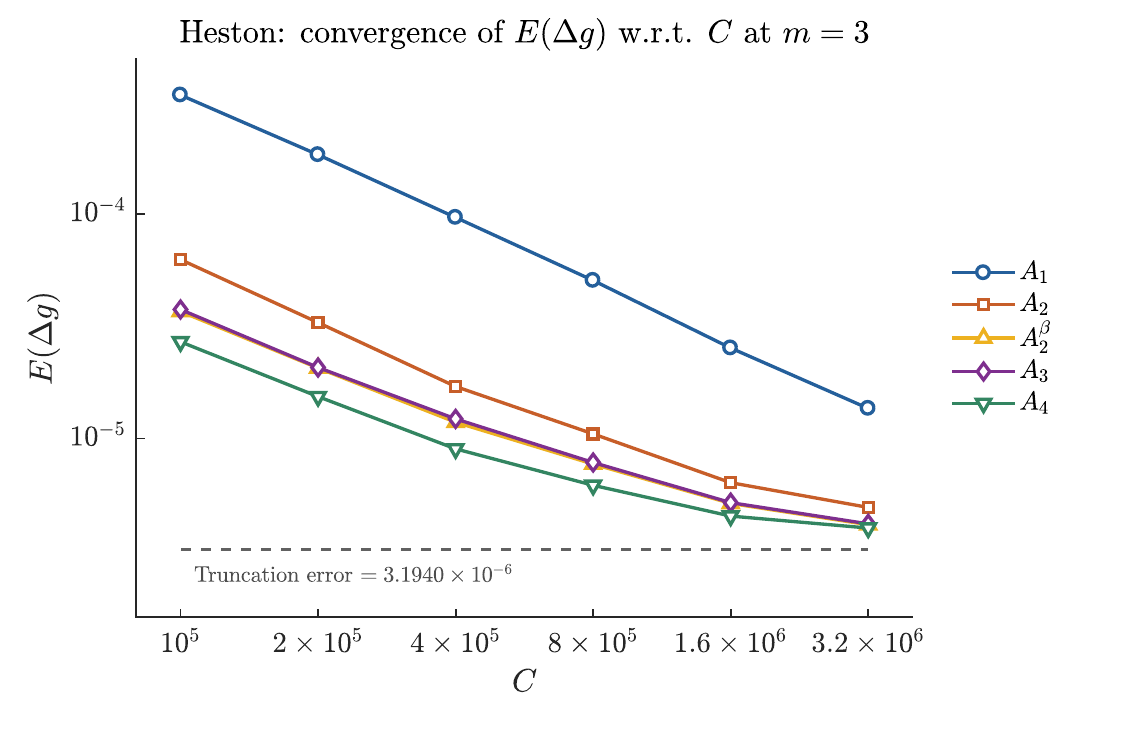}\hfill
		\includegraphics[width=0.5\textwidth]{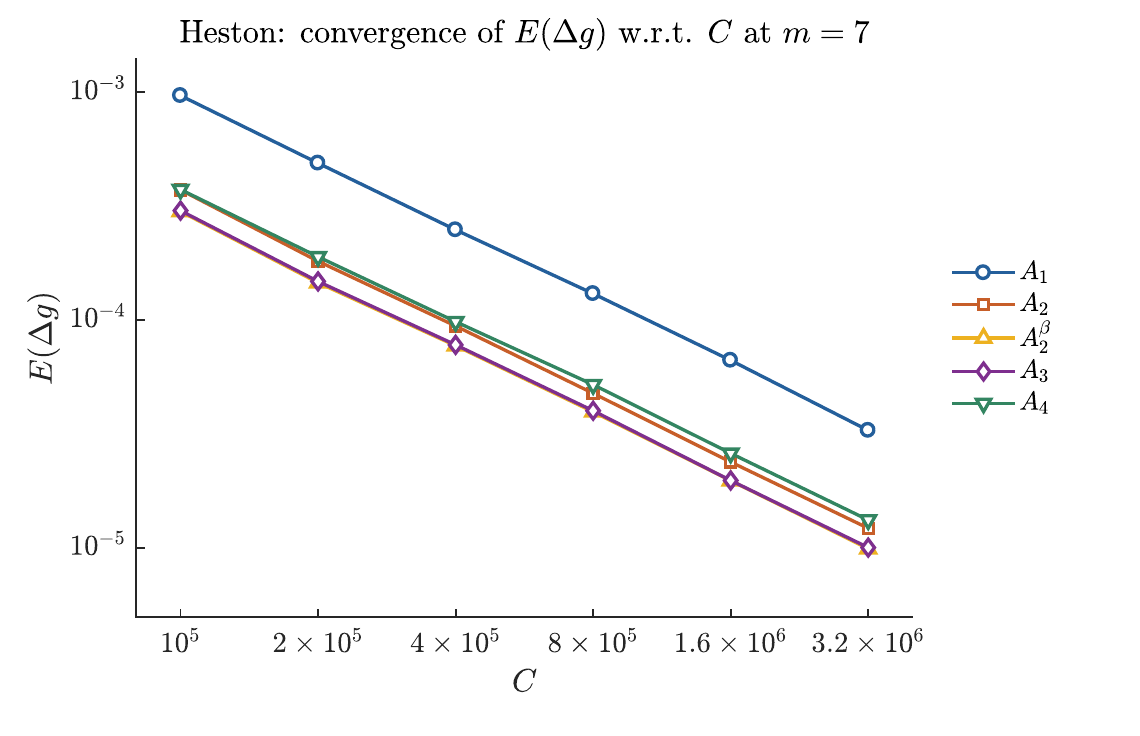}\\[0.5em]
		\caption{Comparison of sample allocation methods for the Heston model. $E(\Delta g)\approx 3.841\times 10^{-2}$ for method $A_0$. Truncation error $\approx 2.14\times 10^{-9}$ when $m=7$.}\label{fig5}
	\end{figure}

	\subsection*{CIR model}
	Finally, we consider the Cox-Ingersoll-Ross (CIR) model \citep{cox1985theory}, which can be expressed as
	\[
	dX(t)=\kappa\bigl(\theta-X(t)\bigr)\,dt+\sigma\sqrt{X(t)}\,dB(t).
	\]
	We set the parameters
	\[
	X(0)=0.04,\quad K=0.03,\quad r=0.05,\quad T=1,\quad
	\kappa=5,\quad \theta=0.04,\quad \sigma=0.2.
	\]
	and the target payoff functional is $Y=e^{-rT}\bigl(X(T)-K\bigr)^+$.

	For numerical approximation of the CIR model, we employ the drift-implicit Milstein scheme proposed by \citep{neuenkirch2014}
	\[
	X_{j+1}^{(n)}
	=
	\frac{X_j^{(n)}+\kappa\theta h
		+\sigma\sqrt{X_j^{(n)}}\,\Delta B_j
		+\dfrac{\sigma^2}{4}\bigl((\Delta B_j)^2-h\bigr)}
	{1+\kappa h},
	\qquad j=0,1,\ldots,2^n-1.
	\]
	where $h=T/2^n$.
	
	The numerical results are presented in Figure \ref{fig3} and Figure \ref{fig6}. All numerical rates are consistent with the theoretical rates, and the performance of these sample allocation methods is similar to the findings for the Black-Scholes model and the Heston model. Unlike the previous two models, we consider $m=5$ rather than $m=3$ in the left panel of Figure \ref{fig6}, since the truncation error dominates when $m=3$, causing the curves for different allocation methods to nearly coincide with the horizontal line representing the truncation error.

\begin{figure}
	\centering
	\includegraphics[width=0.45\textwidth]{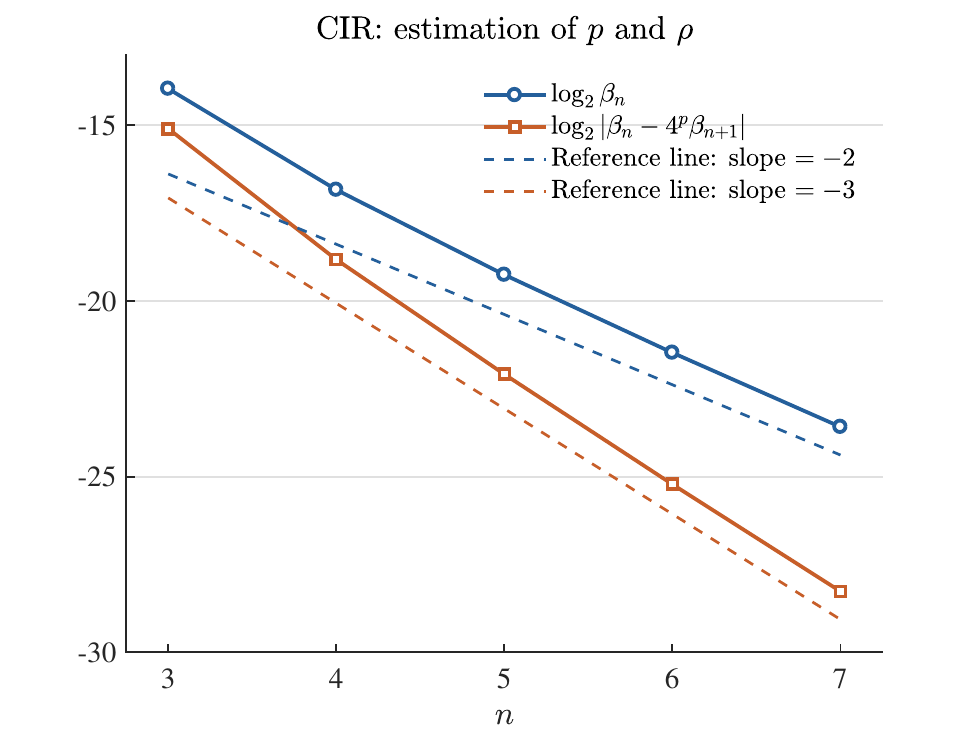}\hfill
	\includegraphics[width=0.45\textwidth]{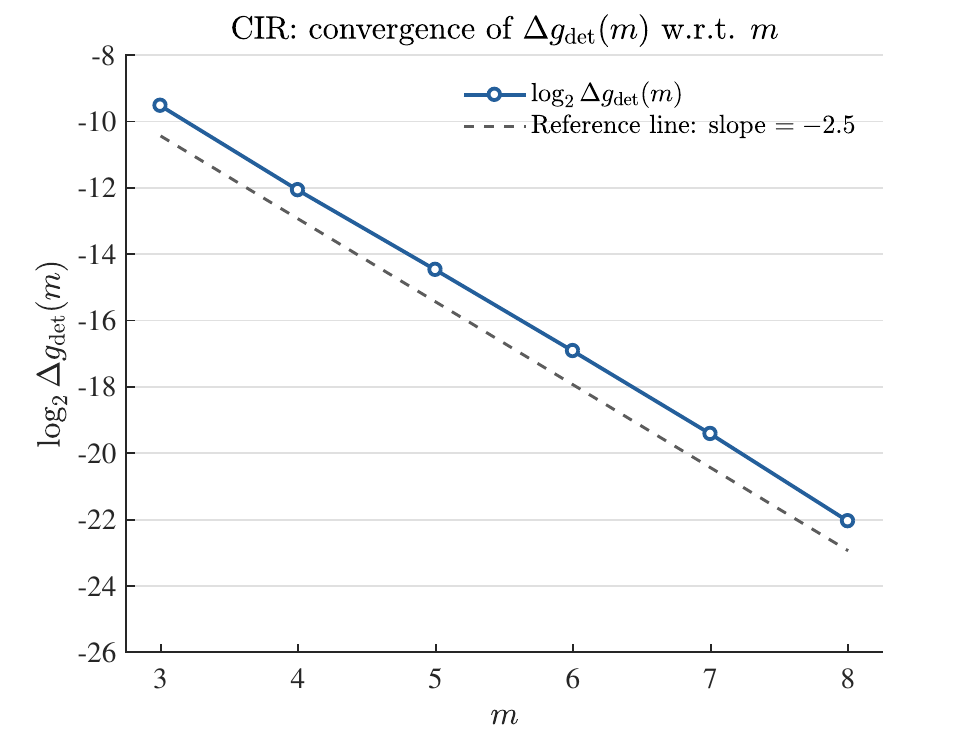}\\[0.5em]
	\includegraphics[width=0.45\textwidth]{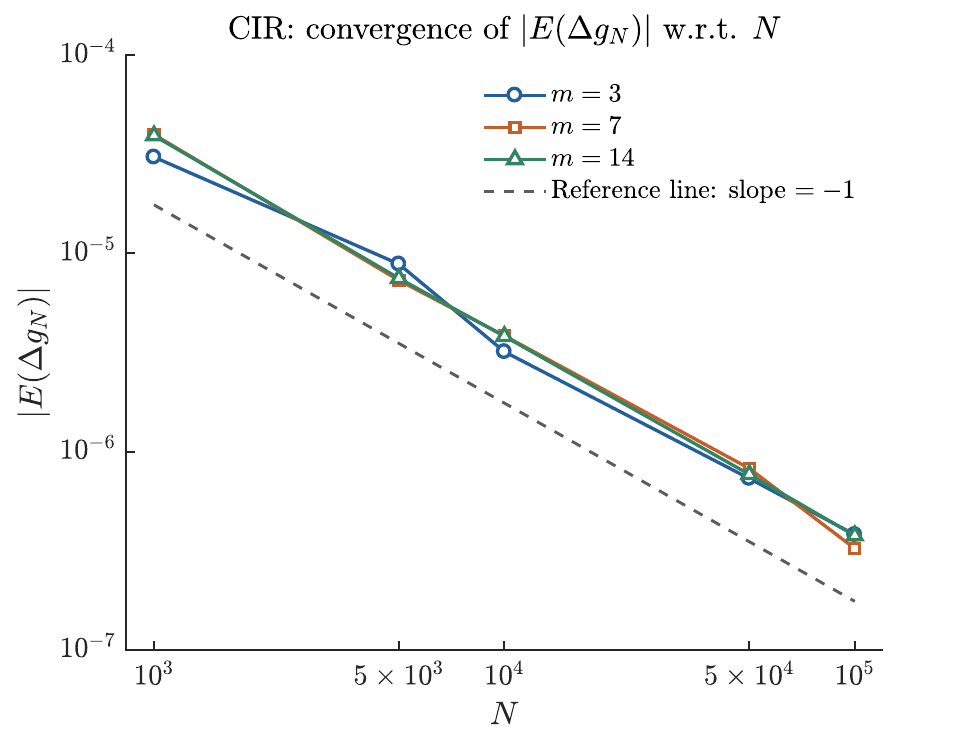}
	\caption{Convergence rates for the CIR model.}\label{fig3}
\end{figure}

	\begin{figure}
		\centering
		\includegraphics[width=0.5\textwidth]{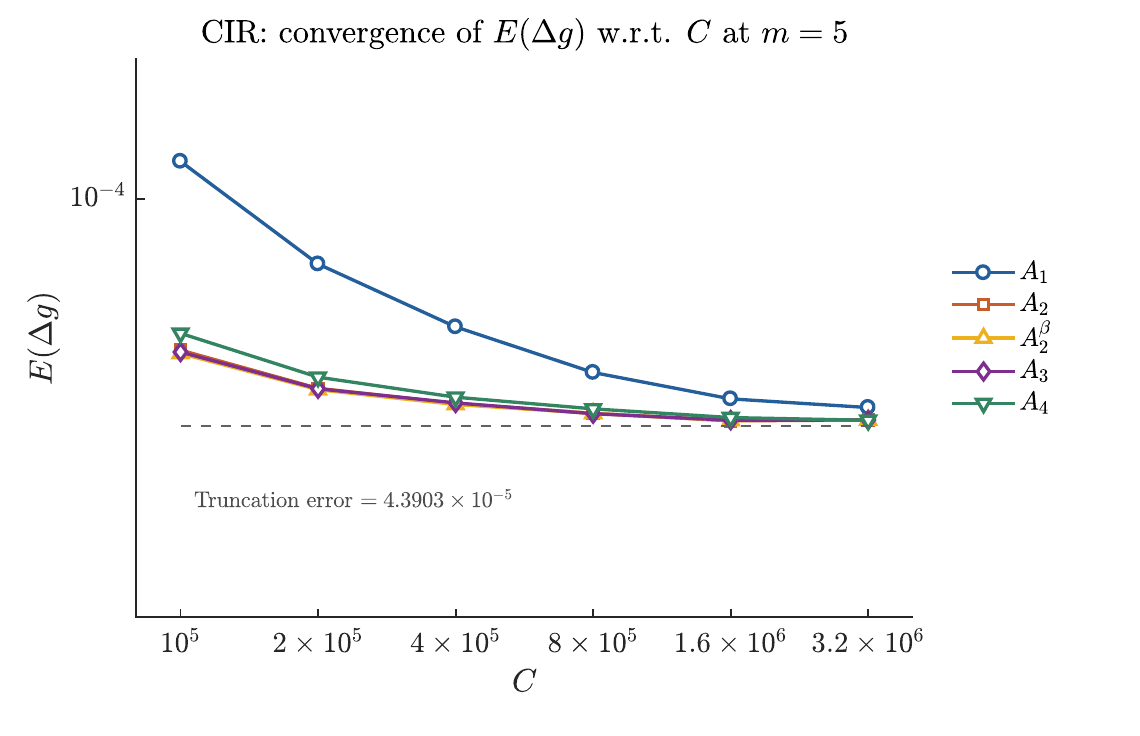}\hfill
		\includegraphics[width=0.5\textwidth]{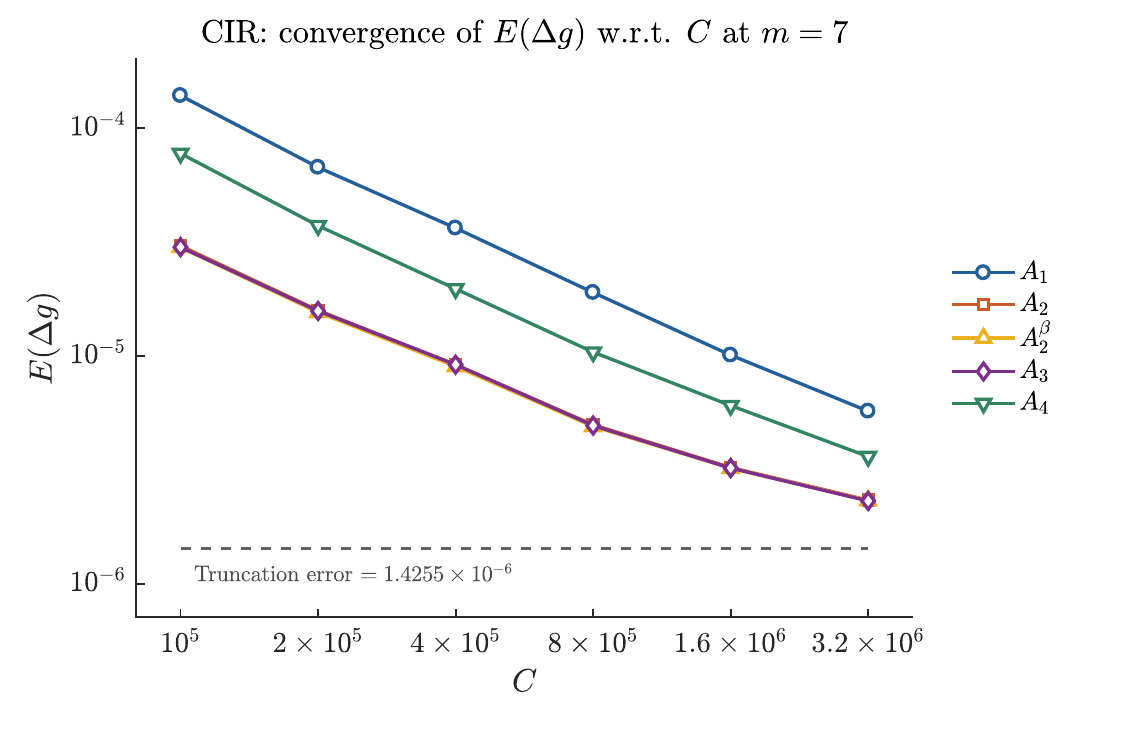}\\[0.5em]
		\caption{Comparison of sample allocation methods for the CIR model. $E(\Delta g)\approx 9.121\times 10^{-3}$ for method $A_0$.}\label{fig6}
	\end{figure}
	
	\section{Conclusion and Discussions}
	
	In this article, we have investigated the computational efficiency loss $\Delta g$ arising from the practical construction of the optimal randomization distribution for the single term estimator. The computational efficiency is measured by the variance-work product, which is the product of two infinite summations. Our analysis accounts for both parameter estimation error and the truncation error introduced by a geometric approximation of the tail distribution. Under mild assumptions on the estimated parameters and Assumption \ref{assu:3.1} that
\[
	\beta_{n} = C_{\beta}\cdot2^{-2pn}(1+O(2^{-\rho n})),\quad p>1/2, \quad \rho>0, \quad t_{n} = 2^n,
\]
we establish the bound for the expected efficiency loss
\[
E(\Delta g)\leq C_1N^{-1}+C_2 2^{-m(p-\frac12+2\eta)},
\]
where the constants are independent of the sample size $N$ and truncation level $m$. Here, $\eta=\min(\rho,1)$.  This bound quantifies the separate contributions of parameter estimation and tail approximation. In particular, cancellation of the leading-order truncation term yields a convergence rate that depends on the non-leading-order decay of $\beta_n$. Consequently, their leading-order decay alone does not fully characterize the effect of truncation on computational efficiency. 

Based on the leading-order parameter estimation error, we have proposed several non-uniform sample allocation strategies that aim to minimizing $E(\Delta g)$ given a fixed computational budget. Numerical experiments for several SDEs models suggest that the Asymptotic allocation method $A_3$ seems to be among the most efficient methods, and it has a simple expression, hence quite easy to implement.  Specifically, for method $A_3$, the sample size $N_i$ at level $i$ is determined by 
    \[
	N_i^{(A_3)}\propto\sqrt{w_i}\,2^{-(2p+1)i/4},\qquad i=0,\ldots,m,
	\]
	where $w_i=\frac{1}{1-2^{-(p-1/2)}}$, if $i=m$, and $w_i=1$, otherwise. The estimate of $\alpha$ is constructed by reusing the samples for $\beta_n$.
	
	Several directions remain for future investigation. First, motivated by the bound on the efficiency loss, it is natural to consider balancing the two sources of error by increasing the sample size $N$ as the truncation level $m$ increases. However, it is not clear how to optimally balance these two sources of error.  Second, a more challenging problem is to determine how the computational budget should be allocated as a function of $m$ for non-uniform allocation strategies.  Finally, another important direction is to extend the error analysis to non-uniform allocation schemes when the estimated parameters are dependent.
	
		\section*{Acknowledgement} The authors used ChatGPT to assist in refining mathematical proofs, performing straightforward calculations, and developing code. The mathematical ideas and proof strategies originated from the authors, who independently wrote the manuscript after carefully reviewing and verifying the AI-assisted outputs. The authors take full responsibility for the content of this manuscript.
	
\section*{Appendix A: Proof of Theorem 5.1}
\label{sec:theorem-proof}

\begin {proof}
For fixed $m>0$, the probabilities satisfy
\begin{equation}\label{eq:normalization}
	p_n=\left(\frac{\beta_n}{\alpha^2+c(m)t_n}\right)^{1/2},
	\qquad \sum_{n=0}^m\ell_np_n=1.
\end{equation}
Let $\theta=(\beta_0,\ldots,\beta_m,\alpha)^\top$,
$G=\nabla p(\theta)$, and $H=\nabla^2g_m(p(m))$, so that
\[
G_{ni}=\frac{\partial p_n}{\partial\theta_i},\qquad
H_{nr}=\frac{\partial^2g_m}{\partial p_n\partial p_r},\qquad
\theta_{m+1}=\alpha.
\]
All derivatives in this section are evaluated at the true parameter vector
$\theta$. 

\subsubsection*{Derivatives of $c(m)$ and $p_n$}
\label{subsec:shared-derivatives}

Differentiating $p_n$ gives
\begin{align}
 \frac{\partial p_n}{\partial\beta_i}
 &=\delta_{ni}\frac{p_n}{2\beta_n}
 -\frac{t_np_n^3}{2\beta_n}\frac{\partial c(m)}{\partial\beta_i},
 \qquad 0\leq i\leq m,\label{eq:first-p-beta}\\
 \frac{\partial p_n}{\partial\alpha}
 &=-\frac{p_n^3}{\beta_n}
 \left(\alpha+\frac{t_n}{2}\frac{\partial c(m)}{\partial\alpha}\right).
 \label{eq:first-p-alpha}
\end{align}
Here $\delta_{ni}$ is the Kronecker delta.
Differentiating the normalization condition in
\eqref{eq:normalization} yields
\[
 0=\frac{\ell_ip_i}{2\beta_i}
       -\frac Q2\frac{\partial c(m)}{\partial\beta_i},\qquad
 0=-\alpha\sum_{n=0}^m\ell_n\frac{p_n^3}{\beta_n}
       -\frac Q2\frac{\partial c(m)}{\partial\alpha}.
\]
Since $Q>0$, it follows that
\begin{equation}\label{eq:first-c}
 \frac{\partial c(m)}{\partial\beta_i}=\frac{\ell_ip_i}{\beta_iQ},
 \qquad
 \frac{\partial c(m)}{\partial\alpha}
 =-\frac{2\alpha}{Q}\sum_{n=0}^m\ell_n\frac{p_n^3}{\beta_n}.
\end{equation}
Thus the entries of $G$ are
\begin{align}
 G_{ni}&=\frac{p_i}{2\beta_i}
 \left(\delta_{ni}-\frac{\ell_it_np_n^3}{\beta_nQ}\right),
 \qquad 0\leq i\leq m,\label{eq:G-beta}\\
 G_{n,m+1}&=-\alpha\frac{p_n^3}{\beta_n}
 \left(1-\frac{t_n}{Q}\sum_{r=0}^m\ell_r\frac{p_r^3}{\beta_r}\right).
 \label{eq:G-alpha}
\end{align}

For the bounds in Lemma \ref{lem4.3}, we also need the second derivatives. Differentiation of $Q$ gives
\begin{align}
 \frac{\partial Q}{\partial\beta_j}
 &=\frac{\ell_jt_jp_j^3}{2\beta_j^2}
 -\frac32\left(\sum_{r=0}^m\ell_r\frac{t_r^2p_r^5}{\beta_r^2}\right)
       \frac{\partial c(m)}{\partial\beta_j},\label{eq:Q-beta}\\
 \frac{\partial Q}{\partial\alpha}
 &=-3\sum_{r=0}^m\ell_r\frac{t_rp_r^5}{\beta_r^2}
 \left(\alpha+\frac{t_r}{2}\frac{\partial c(m)}{\partial\alpha}\right).
 \label{eq:Q-alpha}
\end{align}
Differentiating \eqref{eq:first-c} and substituting
\eqref{eq:first-p-beta}--\eqref{eq:Q-alpha}, we obtain
\begin{align}
 \frac{\partial^2c(m)}{\partial\beta_i\partial\beta_j}
 &=-\frac{\delta_{ij}}{2\beta_i}
      \frac{\partial c(m)}{\partial\beta_i}\notag\\
 &\quad+
 \frac{\partial c(m)}{\partial\beta_i}
 \frac{\partial c(m)}{\partial\beta_j}
 \left(\frac{3}{2Q}\sum_{r=0}^m\ell_r\frac{t_r^2p_r^5}{\beta_r^2}
 -\frac{t_ip_i^2}{2\beta_i}-\frac{t_jp_j^2}{2\beta_j}\right),
 \label{eq:c-betabeta}\\
 \frac{\partial^2c(m)}{\partial\beta_i\partial\alpha}
 &=\frac{\partial c(m)}{\partial\beta_i}
 \Bigg[-\frac{p_i^2}{\beta_i}
 \left(\alpha+\frac{t_i}{2}\frac{\partial c(m)}{\partial\alpha}\right)
 \notag\\
 &\hspace{5em}+\frac3Q\sum_{r=0}^m\ell_r\frac{t_rp_r^5}{\beta_r^2}
 \left(\alpha+\frac{t_r}{2}\frac{\partial c(m)}{\partial\alpha}\right)
 \Bigg].\label{eq:c-betaalpha}
\end{align}
The diagonal case $i=j$ of \eqref{eq:c-betabeta} is
\begin{equation}\label{eq:c-diagonal}
 \frac{\partial^2c(m)}{\partial\beta_i^2}
 =-\frac{\ell_ip_i}{2\beta_i^2Q}
 +\frac{\ell_i^2p_i^2}{\beta_i^2Q^2}
 \left(\frac3{2Q}\sum_{r=0}^m\ell_r\frac{t_r^2p_r^5}{\beta_r^2}
       -\frac{t_ip_i^2}{\beta_i}\right).
\end{equation}
Differentiating the normalization condition twice with respect to
$\alpha$, and using \eqref{eq:first-c}, gives
\begin{equation}\label{eq:c-alphaalpha}
 \frac{\partial^2c(m)}{\partial\alpha^2}
 =\frac2Q\left[
 -\sum_{n=0}^m\ell_n\frac{p_n^3}{\beta_n}
 +3\alpha^2\sum_{n=0}^m\ell_n\frac{p_n^5}{\beta_n^2}
 \left(1-\frac{t_n}{Q}\sum_{r=0}^m\ell_r\frac{p_r^3}{\beta_r}\right)^2
 \right].
\end{equation}

Differentiating \eqref{eq:first-p-beta} and
\eqref{eq:first-p-alpha} once more gives all second derivatives of $p_n$:
\begin{align}
 \frac{\partial^2p_n}{\partial\beta_i\partial\beta_j}
 &=-\delta_{ni}\delta_{nj}\frac{p_n}{4\beta_n^2}
 -\frac{t_np_n^3}{4\beta_n^2}
 \left(\delta_{ni}\frac{\partial c(m)}{\partial\beta_j}
       +\delta_{nj}\frac{\partial c(m)}{\partial\beta_i}\right)
 \notag\\
 &\quad+\frac{3t_n^2p_n^5}{4\beta_n^2}
 \frac{\partial c(m)}{\partial\beta_i}
 \frac{\partial c(m)}{\partial\beta_j}
 -\frac{t_np_n^3}{2\beta_n}
 \frac{\partial^2c(m)}{\partial\beta_i\partial\beta_j},
 \label{eq:p-betabeta}\\
 \frac{\partial^2p_n}{\partial\beta_i\partial\alpha}
 &=-\delta_{ni}\frac{p_n^3}{2\beta_n^2}
 \left(\alpha+\frac{t_n}{2}\frac{\partial c(m)}{\partial\alpha}\right)
 \notag\\
 &\quad+\frac{3t_np_n^5}{2\beta_n^2}
 \frac{\partial c(m)}{\partial\beta_i}
 \left(\alpha+\frac{t_n}{2}\frac{\partial c(m)}{\partial\alpha}\right)
 -\frac{t_np_n^3}{2\beta_n}
 \frac{\partial^2c(m)}{\partial\beta_i\partial\alpha},
 \label{eq:p-betaalpha}\\
 \frac{\partial^2p_n}{\partial\alpha^2}
 &=-\frac{p_n^3}{\beta_n}
 +\frac{3p_n^5}{\beta_n^2}
 \left(\alpha+\frac{t_n}{2}\frac{\partial c(m)}{\partial\alpha}\right)^2
 -\frac{t_np_n^3}{2\beta_n}\frac{\partial^2c(m)}{\partial\alpha^2}.
 \label{eq:p-alphaalpha}
\end{align}
In particular, the diagonal case of \eqref{eq:p-betabeta} reads
\begin{align}
 (\nabla^2p_n)_{ii}
 =\frac{\partial^2p_n}{\partial\beta_i^2}
 &=-\delta_{ni}\frac{p_i}{4\beta_i^2}
 \left(1+2t_ip_i^2\frac{\partial c(m)}{\partial\beta_i}\right)
 \notag\\
 &\quad+\frac{3t_n^2p_n^5}{4\beta_n^2}
 \left(\frac{\partial c(m)}{\partial\beta_i}\right)^2
 -\frac{t_np_n^3}{2\beta_n}\frac{\partial^2c(m)}{\partial\beta_i^2},
 \qquad 0\leq i\leq m.
 \label{eq:p-diagonal}
\end{align}
For $i=m+1$, $(\nabla^2p_n)_{m+1,m+1}$ is given by
\eqref{eq:p-alphaalpha}. 

\subsubsection*{Derivatives of $g_m$}

From $g_m=V_m T_m$,
\begin{equation}\label{eq:gradient}
 \frac{\partial g_m}{\partial p_n}
 =w_n\left(t_nV_m-\frac{\beta_nT_m}{p_n^2}\right).
\end{equation}
A further differentiation yields
\begin{equation}\label{eq:H}
 H_{nr}=2w_n\frac{\beta_nT_m}{p_n^3}\delta_{nr}
 -w_nw_r\left(\frac{t_n\beta_r}{p_r^2}
                    +\frac{t_r\beta_n}{p_n^2}\right).
\end{equation}
Consequently, for $x=(x_0,\ldots,x_m)^\top$,
\begin{equation}\label{eq:quadratic}
 x^\top Hx=2T_m\sum_{n=0}^m w_n\frac{\beta_n}{p_n^3}x_n^2
 -2\left(\sum_{n=0}^m w_n\frac{\beta_n}{p_n^2}x_n\right)
   \left(\sum_{n=0}^m w_nt_nx_n\right).
\end{equation}

\subsubsection*{Calculation of $K_i$}

For $0\leq i\leq m$, equation \eqref{eq:G-beta} gives
\begin{align*}
 \sum_{n=0}^m w_n\frac{\beta_n}{p_n^3}G_{ni}^2
 &=\frac{p_i^2}{4\beta_i^2}
 \left(\frac{w_i\beta_i}{p_i^3}-\frac{2\ell_iw_it_i}{Q}
 +\frac{\ell_i^2}{Q^2}\sum_{r=0}^m w_r\frac{t_r^2p_r^3}{\beta_r}\right),\\
 \sum_{n=0}^m w_n\frac{\beta_n}{p_n^2}G_{ni}
 &=\frac{p_i}{2\beta_i}
 \left(\frac{w_i\beta_i}{p_i^2}-\frac{\ell_iT_m}{Q}\right),\\
 \sum_{n=0}^m w_nt_nG_{ni}
 &=\frac{p_i}{2\beta_i}
 \left(w_it_i-\frac{\ell_i}{Q}\sum_{r=0}^m w_r\frac{t_r^2p_r^3}{\beta_r}\right).
\end{align*}
Substitution into \eqref{eq:quadratic}, we obtain
\begin{equation}\label{eq:GHG-beta}
 (G^\top HG)_{ii}
 =\frac{w_ip_i^2}{2\beta_i^2}
 \left[\frac{\beta_iT_m}{p_i^3}-\frac{w_i\beta_it_i}{p_i^2}
 +\frac{\ell_i}{Q}\left(
 \frac{\beta_i}{p_i^2}\sum_{n=0}^m w_n\frac{t_n^2p_n^3}{\beta_n}
 -T_mt_i\right)\right].
\end{equation}

For the parameter $\alpha$, equation \eqref{eq:G-alpha} yields
\begin{align*}
 \sum_{n=0}^m w_n\frac{\beta_n}{p_n^3}G_{n,m+1}^2
 &=\alpha^2\sum_{n=0}^m w_n\frac{p_n^3}{\beta_n}
 \left(1-\frac{t_n}{Q}\sum_{r=0}^m\ell_r\frac{p_r^3}{\beta_r}\right)^2,\\
 \sum_{n=0}^m w_n\frac{\beta_n}{p_n^2}G_{n,m+1}
 &=-\alpha\left(\sum_{n=0}^m w_np_n
       -\frac{T_m}{Q}\sum_{r=0}^m\ell_r\frac{p_r^3}{\beta_r}\right),\\
 \sum_{n=0}^m w_nt_nG_{n,m+1}
 &=-\alpha\Bigg[\sum_{n=0}^m w_n\frac{t_np_n^3}{\beta_n}-\frac1Q
 \left(\sum_{r=0}^m\ell_r\frac{p_r^3}{\beta_r}\right)
 \left(\sum_{n=0}^m w_n\frac{t_n^2p_n^3}{\beta_n}\right)\Bigg].
\end{align*}
The relation $\beta_n/p_n^2=\alpha^2+c(m)t_n$ implies
\begin{align*}
 \alpha^2\sum_{n=0}^m\ell_n\frac{p_n^3}{\beta_n}+c(m)Q&=1,\\
 \sum_{n=0}^m w_np_n
 &=\alpha^2\sum_{n=0}^m w_n\frac{p_n^3}{\beta_n}
   +c(m)\sum_{n=0}^m w_n\frac{t_np_n^3}{\beta_n},\\
 T_m&=\alpha^2\sum_{n=0}^m w_n\frac{t_np_n^3}{\beta_n}
   +c(m)\sum_{n=0}^m w_n\frac{t_n^2p_n^3}{\beta_n}.
\end{align*}
Using these identities to simplify \eqref{eq:quadratic}, we obtain
\begin{equation}\label{eq:GHG-alpha}
 (G^\top HG)_{m+1,m+1}
 =\frac{2\alpha^2}{Q}\left[
 \left(\sum_{n=0}^m w_n\frac{p_n^3}{\beta_n}\right)
 \left(\sum_{n=0}^m w_n\frac{t_n^2p_n^3}{\beta_n}\right)
 -\left(\sum_{n=0}^m w_n\frac{t_np_n^3}{\beta_n}\right)^2
 \right].
\end{equation}
Finally, by the definition of $K_i$ and
\eqref{eq:gradient},
\begin{align}
 K_i&=(G^\top HG)_{ii}
       +\sum_{n=0}^m\frac{\partial g_m(p(m))}{\partial p_n}
                         (\nabla^2p_n)_{ii}\notag\\
 &=(G^\top HG)_{ii}
       +\sum_{n=0}^m w_n\left(t_nV_m-\frac{\beta_nT_m}{p_n^2}\right)
                         (\nabla^2p_n)_{ii},
 \qquad 0\leq i\leq m+1.
 \label{eq:K}
\end{align}
Substituting \eqref{eq:p-alphaalpha}, \eqref{eq:p-diagonal}, 
\eqref{eq:GHG-beta}, and \eqref{eq:GHG-alpha} gives the expressions in Theorem \ref{the:5.1}.
\end{proof}

	\section*{Appendix B: Second-order derivative bounds for Lemma 4.3}
\label{sec:derivative-bounds}

We use the identities in Appendix A to derive the uniform bounds required
in Lemma~\ref{lem4.3}. Recall that $t_n=2^n$, $p>1/2$,
$\sigma_n^2=\operatorname{Var}(\Delta_n^2)$, and
$\sigma_{m+1}=\sigma_\alpha$, where
$\sigma_\alpha^2=\operatorname{Var}(Y)$.
All constants denoted by $C$ are independent of $m$ and $N$.

Define
\begin{equation*}
	U_i=\left[\theta_i-\frac{\sigma_iB_N}{\sqrt N},
	\theta_i+\frac{\sigma_iB_N}{\sqrt N}\right],
	\quad 0\leq i\leq m+1,\qquad
	U=U_0\times\cdots\times U_{m+1}.
\end{equation*}
Let $\widehat\theta=(\widehat\beta_0,\ldots,\widehat\beta_m,\widehat\alpha)^\top \in U$. We evaluate the identities in Appendix A at $\widehat{\theta}$, replacing $\beta_i$ and $\alpha$ by $\widehat\beta_i$ and $\widehat\alpha$, respectively. The quantities $\sigma_i$ remain the standard deviations associated with the true parameters.

By the argument in Lemma 4.2, the following estimates
hold uniformly over $\widehat\theta\in U$ for all sufficiently large $N$:
\begin{equation}\label{eq:uniform-basic}
 \widehat\beta_n\asymp2^{-2pn},\qquad
 p_n(\widehat\theta)\asymp2^{-(p+1/2)n},\qquad
 \widehat\alpha^2+c(m)t_n
       =\frac{\widehat\beta_n}{p_n(\widehat\theta)^2}\asymp2^n.
\end{equation}
Here and below, $c(m)$, $p_n$ and $Q$ are evaluated at $\widehat\theta$. It follows from \eqref{eq:uniform-basic} that
\begin{equation}\label{eq:building-blocks}
 \frac{p_n}{\widehat\beta_n}\leq C2^{(p-1/2)n},\quad
 \frac{p_n}{\widehat\beta_n^2}\leq C2^{(3p-1/2)n},\quad
 \frac{p_n^3}{\widehat\beta_n}\leq C2^{-(p+3/2)n},\quad
 \frac{p_n^5}{\widehat\beta_n^2}\leq C2^{-(p+5/2)n}.
\end{equation}
The $n=0$ term gives a positive lower bound for $Q$, while
\eqref{eq:uniform-basic} gives the upper bounds
\begin{equation}\label{eq:uniform-sums}
 C^{-1}\leq Q=\sum_{n=0}^m\ell_n
                       \frac{t_n p_n^3}{\widehat\beta_n}\leq C,
 \qquad
 \sum_{n=0}^m\ell_n\frac{p_n^3}{\widehat\beta_n}\leq C,
 \qquad
 \sum_{n=0}^m\ell_n\frac{t_n^2p_n^5}{\widehat\beta_n^2}\leq C.
\end{equation}

\subsubsection*{Derivative bounds for $c(m)$}

Equations \eqref{eq:first-c} and \eqref{eq:uniform-sums} imply
\begin{equation}\label{eq:c-first-bounds}
 \left|\frac{\partial c(m)}{\partial\widehat\beta_i}\right|
 \leq C2^{(p-1/2)i},\qquad
 \left|\frac{\partial c(m)}{\partial\widehat\alpha}\right|\leq C.
\end{equation}
Moreover,
\[
 \frac{t_ip_i^2}{\widehat\beta_i}\leq C,\qquad
 \left|\widehat\alpha+\frac{t_i}{2}
                   \frac{\partial c(m)}{\partial\widehat\alpha}\right|
 \leq C2^i.
\]
It follows from \eqref{eq:c-betabeta} that
\begin{align}
 \left|\frac{\partial^2c(m)}
       {\partial\widehat\beta_i\partial\widehat\beta_j}\right|
 &\leq C2^{(p-1/2)i+(p-1/2)j},\qquad i\ne j,
 \label{eq:c-mixed-bound}\\
 \left|\frac{\partial^2c(m)}{\partial\widehat\beta_i^2}\right|
 &\leq C\left(2^{(3p-1/2)i}+2^{(2p-1)i}\right)
 \leq C2^{(3p-1/2)i}.
 \label{eq:c-diagonal-bound}
\end{align}
Both terms in the brackets in \eqref{eq:c-betaalpha} are uniformly
bounded, since
\[
 \frac{p_i^2}{\widehat\beta_i}
 \left|\widehat\alpha+\frac{t_i}{2}
              \frac{\partial c(m)}{\partial\widehat\alpha}\right|\leq C
\]
and
\[
 \sum_{n=0}^m\ell_n\frac{t_np_n^5}{\widehat\beta_n^2}
 \left|\widehat\alpha+\frac{t_n}{2}
              \frac{\partial c(m)}{\partial\widehat\alpha}\right|
 \leq C\sum_{n=0}^m2^{-(p+1/2)n}\leq C.
\]
We also have
\[
\widehat\alpha^2\sum_{n=0}^m\ell_n\frac{p_n^5}{\widehat\beta_n^2}
	\left(1-\frac{t_n}{Q}\sum_{r=0}^m\ell_r\frac{p_r^3}{\widehat\beta_r}\right)^2\leq C.
\]
From \eqref{eq:c-betaalpha}, \eqref{eq:c-alphaalpha} and the bounds above, we obtain
\begin{equation}\label{eq:c-alpha-bounds}
 \left|\frac{\partial^2c(m)}
       {\partial\widehat\beta_i\partial\widehat\alpha}\right|
 \leq C2^{(p-1/2)i},\qquad
 \left|\frac{\partial^2c(m)}{\partial\widehat\alpha^2}\right|\leq C.
\end{equation}

\subsubsection*{Derivative bounds for $p_n$}

The first-order identities \eqref{eq:first-p-beta} and \eqref{eq:first-p-alpha} give the bounds used in Lemma 4.3:
\begin{align*}
 \sup_{\widehat\theta\in U}
 \left|\frac{\partial p_n}{\partial\widehat\alpha}\right|
 &\leq C2^{-(p+1/2)n},\qquad
 \sup_{\widehat\theta\in U}
 \left|\frac{\partial p_n}{\partial\widehat\beta_n}\right|
 \leq C2^{(p-1/2)n},\notag\\
 \sup_{\widehat\theta\in U}
 \left|\frac{\partial p_n}{\partial\widehat\beta_k}\right|
 &\leq C2^{-(p+1/2)n+(p-1/2)k},\qquad k\ne n.
\end{align*}
Substituting \eqref{eq:building-blocks}--\eqref{eq:c-alpha-bounds}
into \eqref{eq:p-betabeta}--\eqref{eq:p-alphaalpha} gives
\begin{align*}
 \sup_{\widehat\theta\in U}
 \left|\frac{\partial^2p_n}{\partial\widehat\alpha^2}\right|
 &\leq C2^{-(p+1/2)n},\\
 \sup_{\widehat\theta\in U}
 \left|\frac{\partial^2p_n}
 {\partial\widehat\beta_n\partial\widehat\alpha}\right|
 &\leq C2^{(p-1/2)n},\\
 \sup_{\widehat\theta\in U}
 \left|\frac{\partial^2p_n}
 {\partial\widehat\beta_k\partial\widehat\alpha}\right|
 &\leq C2^{-(p+1/2)n+(p-1/2)k},\qquad k\ne n,
\\
 \sup_{\widehat\theta\in U}
 \left|\frac{\partial^2p_n}{\partial\widehat\beta_n^2}\right|
 &\leq C2^{(3p-1/2)n},\\
 \sup_{\widehat\theta\in U}
 \left|\frac{\partial^2p_n}
 {\partial\widehat\beta_n\partial\widehat\beta_k}\right|
 &\leq C2^{(p-1/2)n+(p-1/2)k},\qquad k\ne n,\\
 \sup_{\widehat\theta\in U}
 \left|\frac{\partial^2p_n}{\partial\widehat\beta_k^2}\right|
 &\leq C2^{-(p+1/2)n+(3p-1/2)k},\qquad k\ne n,
\end{align*}
and
\[
 \sup_{\widehat\theta\in U}
 \left|\frac{\partial^2p_n}
 {\partial\widehat\beta_k\partial\widehat\beta_j}\right|\leq C2^{-(p+1/2)n+(p-1/2)k+(p-1/2)j}, \qquad k,j\ne n,\quad k\ne j.
\]
This establishes all seven second-order derivative bounds.

	\bibliographystyle{plainnat}
	\bibliography{reference}
\end{document}